\documentclass[11pt]{amsart}

\usepackage{amsmath,amsthm,amssymb,color}

\usepackage{pdfsync}
\usepackage[latin1]{inputenc}
\usepackage{graphicx}
\usepackage{pstricks,epic,eepic}

\newcommand{\C}{{\mathbb C}}       %
\newcommand{\R}{{\mathbb R}}       %

\newcommand{\DD}{{\mathcal D}}
\newcommand{\HH}{{\mathcal H}}
\newcommand{\LL}{{\mathcal L}}

\newcommand{\AZ}{{\mathcal A}}

\newcommand{\RR}{{\mathcal R}}
\newcommand{\FF}{{\mathcal F}}

\newcommand{\EE}{{\mathcal E}}

\newcommand{\MD}{{\mathcal {MD}}}

\newcommand{\CC}{{\mathcal C}}

\newcommand{\diam}{\operatorname{diam}}
\newcommand{\dist}{{\rm dist}}

\newcommand{\rf}[1]{{\eqref{#1}}}
\newcommand{\supp}{\operatorname{supp}}

\newcommand{\ve}{{\varepsilon}}
\newcommand{\vv}{{\vspace{2mm}}}
\newcommand{\vvv}{{\vspace{3mm}}}
\newcommand{\wt}[1]{{\widetilde{#1}}}
\newcommand{\wh}[1]{{\widehat{#1}}}

\newcommand{\LD}{{\mathsf{LD}}}
\newcommand{\HD}{{\mathsf{HD}}}

\newcommand{\sss}{{\mathsf{Stop}}}

\newcommand{\nex}{{\mathsf{Next}}}

\newcommand{\ttt}{{\mathsf{Top}}}

\newcommand{\good}{{\mathsf{Good}}}

\newcommand{\tree}{{\mathsf{Tree}}}
\newcommand{\tr}{{\mathsf{Tr}}}

\newcommand{\BCE}{{\mathsf{BCE}}}

\newcommand{\sep}{{\mathsf{Sep}}}

\newtheorem{theorem}{Theorem}[section]
\newtheorem*{theorem*}{Theorem}
\newtheorem*{lemma*}{Lemma}
\newtheorem*{theorema*}{Theorem A}
\newtheorem*{theoremb*}{Theorem B}
\newtheorem*{theoremc*}{Theorem C}
\newtheorem*{mainlemma*}{Main Lemma}

\newtheorem{lemma}[theorem]{Lemma}

\newtheorem{coro}[theorem]{Corollary}
\newtheorem{propo}[theorem]{Proposition}

\theoremstyle{definition}

\def\XXint#1#2#3{{\setbox0=\hbox{$#1{#2#3}{\int}$ }
\vcenter{\hbox{$#2#3$ }}\kern-.58\wd0}}

\theoremstyle{remark}
\newtheorem{remark}[theorem]{\bf Remark}

\numberwithin{equation}{section}

\begin{document}

\title{Analytic capacity and projections}

\author{Alan Chang}
\author{Xavier Tolsa}

\address{Alan Chang\\
Department of Mathematics\\
University of Chicago, Chicago, IL 60637, USA}

\email{ac@math.uchicago.edu}

\address{Xavier Tolsa
\\
ICREA, Passeig Llu\'{\i}s Companys 23 08010 Barcelona, Catalonia, and\\
Departament de Matem\`atiques and BGSMath
\\
Universitat Aut\`onoma de Barcelona
\\
08193 Bellaterra (Barcelona), Catalonia
}

\email{xtolsa@mat.uab.cat}

\thanks{A.C. was supported by the National Science Foundation grants DGE-1144082, DGE-1746045, and DMS-1246999. X. T. was supported by the ERC grant 320501 of the European Research Council (FP7/2007-2013) and partially supported by MTM-2016-77635-P,  MDM-2014-044 (MICINN, Spain), and by Marie Curie ITN MAnET (FP7-607647).}

\begin{abstract}
In this paper we study the connection between the analytic capacity of a set and the size of its orthogonal projections. More precisely, we prove that if $E\subset \C$ is compact and $\mu$ is a Borel measure supported on $E$, then
 the analytic capacity of $E$ satisfies
$$
\gamma(E) \geq c\,\frac{\mu(E)^2}{\int_I \|P_\theta\mu\|_2^2\,d\theta},
$$
where $c$ is some positive constant, $I\subset [0,\pi)$ is an arbitrary interval, and $P_\theta\mu$ is the image measure of
$\mu$ by $P_\theta$, the orthogonal projection onto the line $\{re^{i\theta}:r\in\R\}$. This result is
related to an old conjecture of Vitushkin about the relationship between the Favard length and analytic capacity.
We also prove a generalization of the above inequality to higher dimensions which involves related capacities associated with
signed Riesz kernels.
\end{abstract}

\maketitle

\section{Introduction}

The objective of this paper is to study the connection between the analytic capacity of a set and the size of its projections
 onto lines. First we introduce some notation and definitions.
A compact set $E\subset \C$ is said to be {\em removable for bounded analytic
functions} if for any open set $\Omega$ containing $E$, every
bounded function analytic on $\Omega\setminus E$ has an analytic
extension to $\Omega$. In order to study removability,  in the 1940s, Ahlfors \cite{Ahlfors} 
introduced the notion of analytic capacity.
The {\em analytic capacity} of a compact set $E\subset\C$ is
\begin{equation}\label{defgam}\gamma(E) = \sup|f'(\infty)|,
\end{equation}
where the supremum is taken over all analytic functions
$f:\C\setminus E \to \C$ with $|f|\leq1$ on $\C\setminus E$, and
$f'(\infty)=\lim_{z\to\infty} z(f(z)-f(\infty))$.
In \cite{Ahlfors}, Ahlfors showed that $E$ is removable for bounded analytic functions if and only if 
$\gamma(E)=0$.

In the 1960s, Vitushkin conjectured that a compact set in the plane is non-removable for bounded analytic functions (or equivalently, has positive analytic capacity) if and only if its orthogonal projections have positive length in a set of directions of positive measure, or in other words, if and only if it has positive Favard length.
The {\em Favard length} of a Borel set $E\subset\C$ is
\begin{equation} \label{favard}
{\rm Fav}(E) =\int_0^\pi \HH^1(P_\theta(E))\,d\theta,
\end{equation}
where, for $\theta\in [0,\pi)$,  $P_\theta$ denotes the orthogonal
projection onto the line $L_\theta := \{re^{i\theta}:r\in\R\}$, and $\HH^1$ stands for the $1$-dimensional Hausdorff measure. 
In 1986, Mattila \cite{Mattila-proj} showed that Vitushkin's conjecture is false. 
Indeed, he proved that the 
property of having zero Favard length is not invariant under conformal mappings while removability for bounded
analytic functions remains invariant. Mattila's result didn't tell which implication in the above conjecture
is false. This question was partially solved in 1988, when Jones and Murai \cite{Jones-Murai} 
constructed a set with zero Favard length and positive analytic 
capacity. Later on, Joyce and M\"orters \cite{JyM} later obtained an easier example using curvature of measures.

Although Vitushkin's conjecture is not true in full generality, it turns out that it holds in the particular case of
sets with finite length. This was proved by G.~David \cite{David-vitus} in 1998. 
Indeed, he showed that such sets are removable if and only if they are purely unrectifiable, which is equivalent to having zero Favard length, by the Besicovitch projection theorem.
For sets with arbitrary length, removability can be characterized in terms of curvature of measures, by \cite{Tolsa-sem}
(see Theorem \ref{teogam} below for more details).

As mentioned above, one of the directions of Vitushkin's conjecture is false.
 However, it is not known yet if the other implication holds. Namely, does positive Favard length imply positive analytic capacity?   In a sense, the main result of this paper asserts that if one strengthens the assumption of positive Favard length in a suitable way, then the answer is positive. See the survey \cite{EV} for more information on this and other related questions.

Given a Borel measure $\mu$ in $\R^2$, we denote by
$P_\theta\mu$ the image measure of $\mu$ by the orthogonal projection $P_\theta$ from
$\R^2$ onto the line $L_\theta := \{re^{i\theta}:r\in\R\}$.
Our main result is the following:

\vv

\begin{theorem}\label{teo1}
Let $I\subset [0,\pi)$ be an interval. For any compact set $E\subset \C$ and any Borel measure $\mu$
on $\C$ with $\int_I \|P_\theta\mu\|_2^2\,d\theta > 0$, we have
\begin{equation}\label{eqgamfav}
\gamma(E) \geq c\,\frac{\mu(E)^2}{\int_I \|P_\theta\mu\|_2^2\,d\theta},
\end{equation}
where $c$ is some positive constant depending only on $\HH^1(I)$.
\end{theorem}

In the theorem above,   $\|P_\theta\mu\|_2$ stands for the $L^2$ norm of the density of $P_\theta\mu$
with respect to length in $L_\theta$, and the $L^2$ norm is computed with respect to length in $L_\theta$ too.

Observe that if the Hausdorff dimension of $E$ is larger than $1$, by Frostman's lemma there exists
a nonzero measure $\mu$ supported on $E$ such that $\mu(B(x,r))\leq r^s$, with $s>1$. Then it holds that
\begin{equation}\label{eq**r}
 \int_0^\pi \|P_\theta\mu\|_2^2\,d\theta\leq c\,\iint \frac1{|x-y|}\,d\mu(x)\,d\mu(y)<\infty.
\end{equation} 
  See Theorem 9.7 in \cite{Mattila-gmt}, for example. So \rf{eqgamfav} implies the well-known fact that $\gamma(E)>0$ in this case. In fact, from \rf{eq**r} and \rf{eqgamfav} we get the sharper (also well-known) inequality $\gamma(E)\geq c\, C_1(E)$, where $C_1$ is the capacity associated with the Riesz kernel $1/|x|$. 
  
  On the other hand, there are sets $E$ in the plane with  $C_1(E)=0$ which support a nonzero Borel measure $\mu$ such that
$\int_I \|P_\theta\mu\|_2^2\,d\theta<\infty$ for some interval $I\subset[0,\pi)$. This is the case, for example, of any rectifiable set $E$ with
$\HH^1(E)<\infty$.  To see this, just consider a Lipschitz graph $\Gamma$ such that $\HH^1(E\cap\Gamma)>0$ and let $\mu=\HH^1|_{E\cap\Gamma}$. If $L_{\theta_0}$ denotes the horizontal axis with respect to which $\Gamma$ is constructed and $I$ is a sufficiently small neighborhood of $\theta_0$, then $\int_I \|P_\theta\mu\|_2^2\,d\theta<\infty$, as desired.

It is worth mentioning that the Favard length of $E$ satisfies an estimate very similar to~\rf{eqgamfav}:
\begin{equation}\label{eqfavfav}
{\rm Fav}(E) \geq c_I\,\frac{\mu(E)^2}{\int_I \|P_\theta\mu\|_2^2\,d\theta},
\end{equation}
for some constant $c_I>0$.
This follows easily from the Cauchy-Schwarz inequality. Indeed, for any Borel measure $\mu$ supported on $E$ and
$\theta\in [0,\pi)$,
\begin{equation}\label{eq:cauchyschwarz1}
\mu(E)= \|P_\theta\mu\|_1\leq \|P_\theta\mu\|_2\,\HH^1(P_\theta E)^{1/2}.
\end{equation}
Thus,
$$\mu(E) \,\HH^1(I)\leq \int_I  \|P_\theta\mu\|_2\,\HH^1(P_\theta E)^{1/2}\,d\theta
\leq \left(\int_I  \|P_\theta\mu\|_2^2\,d\theta\right)^{1/2}\,{\rm Fav}(E)^{1/2},$$
which yields \rf{eqfavfav} with $c_I=\HH^1(I)^2$.
So a natural question is the following: does there exist some absolute constant $c>0$ such that
\begin{equation}\label{quest}
\gamma(E) \geq c\,{\rm Fav}(E)\,?
\end{equation}
If the answer were positive, then Theorem \ref{teo1} would be an immediate consequence of this and \rf{eqfavfav}.

The question \rf{quest} is widely open, and we think that
Theorem \ref{teo1} can be interpreted as a contribution that supports a positive answer.
However, we remark that the assumption that ${\rm Fav}(E)>0$ is strictly stronger than the existence of a measure $\mu$ supported on $E$ satisfying $P_\theta\mu\in L^2$ for a.e.\ $\theta$ in some interval $I\subset[0,\pi)$. Indeed, 
there exists a set $E \subset \R^2$ such that the following hold:
\begin{enumerate}
\item $\HH^1(P_\theta E ) > 0$ for a.e.~$\theta \in [0, \pi)$.
\item For all Borel measures $\mu$ such that $\mu(E) > 0$ and all intervals $I \subset [0, \pi)$, we have $\HH^1(\{ \theta \in I : P_\theta \mu \not\in L^2 \}) > 0$.
\end{enumerate}
%

To construct a set $E$ satisfying (1) and (2), let $\{I_k\}_{k=1}^\infty$ be a sequence of subsets of $[0, \pi)$ such that $\HH^1(\bigcap_k I_k) = 0$ and such that $\HH^1(I \cap I_k) > 0$ for all $k$ and all intervals $I \subset [0, \pi)$. (For example, we can take each $I_k$ to be open and dense in $[0, \pi)$, with $\HH^1(I_k) \to 0$.) By the digital sundial theorem \cite{Falconer}, for each $k$, there is a set $E_k \subset \R^2$ such that $\HH^1(P_\theta E_k) = 0$ for a.e.~$\theta \in I_k$ and $\HH^1(P_\theta E_k) > 0$ for a.e.~$\theta \not\in I_k$. Let $E = \bigcup_k E_k$. 

Property (1) immediately follows from the facts that $\HH^1(P_\theta E) > 0$ for a.e.~$\theta \not\in \bigcap_k I_k$ and that $\HH^1(\bigcap_k I_k) = 0$. To check (2), let $\mu$ be a Borel measure such that $\mu(E) > 0$. Then $\mu(E_k) > 0$ for some $k$, so using \eqref{eq:cauchyschwarz1}, we see that if $\theta \in [0, \pi)$ is such that $P_\theta (\mu|_{E_k}) \in L^2$, then $\HH^1(P_\theta E_k) > 0$. Hence, for any interval $I \subset [0, \pi)$, we have
\[
\HH^1(\{ \theta \in I : P_\theta \mu \not\in L^2 \})
\geq
\HH^1(\{ \theta \in I : \HH^1(P_\theta E_k) = 0 \})
=
\HH^1(I_k \cap I)
>
0,
\]
and we have verified (2).

\vv

The result stated in Theorem \ref{teo1} extends to higher dimensions. In $\R^d$
the role of the analytic capacity $\gamma$ is played by the capacities $\Gamma_{d,n}$ associated to the vector-valued Riesz kernels $x/|x|^{n+1}$.
Given an integer $0<n<d$ and a compact $E\subset\R^d$,  one sets
$$\Gamma_{d,n}(E) = \sup|\langle T,1\rangle|,$$
where the supremum is take over all real distributions $T$ supported in $E$ such that 
$$\frac x{|x|^{n+1}} * T \in L^\infty(\R^d)\quad\mbox{ and }\quad \Bigl\|\frac x{|x|^{n+1}} * T\Bigr\|_{L^\infty(\R^d)}\leq 1.$$
In the case $n=1$, $d=2$, $\Gamma_{2,1}$ is the real version of the analytic capacity $\gamma$, and from \cite{Tolsa-sem}
it holds that $\gamma\approx\Gamma_{2,1}$. 

\vv

In the codimension $1$ case (i.e., $n=d-1$), $\Gamma_{d,d-1}$ is the so-called Lipschitz harmonic capacity introduced by Paramonov
\cite{Paramonov}.
The analogue of Vitushkin's conjecture also holds for sets with finite $\HH^n$-measure. That is, $E$ removable for
Lipschitz harmonic functions in $\R^{n+1}$ if and only if $E$ is purely $n$-unrectifiable,
or equivalently, the orthogonal projections of $E$ on almost all hyperplanes have $\HH^n$-measure zero. See \cite{NToV1} and \cite{NToV2}.
The analogous result for $1<n<d-1$ is still an open problem. %

The higher dimensional extension of Theorem \ref{teo1} is the following:

\begin{theorem}\label{teo2}
For an integer $n$ with $0<n<d$,
let $V_0\subset\R^d$ be an $n$-plane through the origin and let $s>0$.
For any compact set $E\subset \R^d$ and any Borel measure $\mu$
on $\R^d$ with $\int_{B(V_0, s)} \| P_V\mu \|_2^2 \, d \gamma_{d,n}(V) > 0$, we have
$$\Gamma_{d,n}(E) \geq c\,\frac{\mu(E)^2}{\int_{B(V_0, s)} \| P_V\mu \|_2^2 \, d \gamma_{d,n}(V)},$$
where $c$ is some positive constant depending only on $s$, $n$, and $d$.
\end{theorem}

In this theorem, $P_V$ is the orthogonal projection onto the $n$-dimensional subspace $V$, and $\| P_V\mu \|_2$ is the $L^2$ norm
(with respect to $n$-dimensional Lebesgue measure) in $V$ of $P_V\mu$ (we identify $P_V\mu$ with its density  with respect 
to $n$-dimensional Lebesgue measure in $V$). Also, $\gamma_{d,n}$ is the natural probability measure on
$G(d,n)$ (see \cite[Chapter 3]{Mattila-gmt}), and $B(V_0,s)$ is a ball of radius $s$ in $G(d,n)$, where $G(d, n)$ denotes the Grassmanian space of $n$-dimensional linear subspaces of $\R^d$. See 
Section \ref{sechighdim} below for the definition of the metric in $G(d,n)$.
\vv

The first fundamental step towards the proof of Theorems \ref{teo1} and \ref{teo2} is a Fourier calculation which shows that
there exist constants $c,\lambda>1$ (depending only on $s, n, d$) such that
\begin{equation}\label{eqclau*1}
 \iint_{x-y \in K(V_0^\bot,\lambda^{-1} s)}\frac{ d\mu (x)\,  d\mu (y)}{|x-y|^n} 
\leq c
\int_{B(V_0, s)} \| P_V\mu \|_2^2 \, d \gamma_{d,n}(V) %
,
\end{equation}
where, given $U\in G(d,m)$ and $t>0$, $K(U,t)$ is the cone
$$K(U,t) = \bigl\{x\in\R^d:\dist(x,U)<t\,|x|\bigr\},$$
and $V_0^\bot$ is the subspace orthogonal to $V_0$.
In the planar case $n=1$, $d=2$, the calculation is particularly clean and we obtain a more precise result.
See Corollaries \ref{corofourier} and \ref{corofourier2} for more details. Our inspiration for proving
the estimate \rf{eqclau*1} 
comes from the work of Martikainen and Orponen \cite{MO}. In that paper, the authors characterize the ``big pieces of
Lipschitz graph'' condition on $n$-AD-regular sets in terms of an integral of the form $\int_{B(V_0, s)} \| P_V\mu \|_2^2 \, d \gamma_{d,n}(V)$, with $\mu$ equal to $n$-dimensional Hausdorff measure $\HH^n$ restricted to a suitable subset $E$.
Roughly speaking, in the proof of the main lemma (Lemma 1.10) of \cite{MO}, the authors obtain a variant of the estimate \eqref{eqclau*1}, but with the left-hand side replaced by a discretized version of the integral.
They do not use the Fourier transform, but instead use more
geometric arguments.

The second step in the proof of Theorems \ref{teo1} and \ref{teo2} is the construction of a corona type decomposition which will allow
us to bound the $L^2(\mu)$ norm of the Riesz transform
$$\RR^n\mu(x) = \int \frac{x-y}{|x-y|^{n+1}}\,d\mu(y),$$
assuming that $\mu$ satisfies the growth condition
$$\mu(B(x,r))\leq c_0\,r^n\quad\mbox{ for all $x\in\R^d$, $r>0$.}
$$
Using this corona decomposition we will deduce that
\begin{equation}\label{eqclau*2}
\|\RR^n\mu\|_{L^2(\mu)}^2 \lesssim \mu(\R^d) + \iint_{x-y \in K(U_0,s)}\!\!\frac{ d\mu (x)\,  d\mu (y)}{|x-y|^n}
\end{equation}
for any $U_0\in G(d,d-n)$ and $s>0$. 
In the case $n=1$, we will get the following estimate involving the curvature of $\mu$:
\begin{equation}\label{eqclau*3}
\iiint \frac1{R(x,y,z)^2}\,d\mu(x)\,d\mu(y)\,d\mu(z) \lesssim \mu(\C) + \iint_{x-y \in K(U_0,s)}\!\!\frac{ d\mu (x)\,  d\mu (y)}{|x-y|},
\end{equation}
where $R(x,y,z)$ stands for the radius of the circumference passing through $x,y,z$. In both \rf{eqclau*2} and \rf{eqclau*3},
the implicit constant depends only on $s$, $c_0$, $n$ and $d$.
See Theorem \ref{teocurv} for more details. We remark that other related corona decompositions have already appeared in 
\cite{Tolsa-bilip} and \cite{Azzam-Tolsa-GAFA}, for example. However, the use of the conical Riesz-type energy in \rf{eqclau*1} in connection with corona decompositions is totally
new, as far as we know.

Using \rf{eqclau*1}, \rf{eqclau*2}, \rf{eqclau*3}, and the characterization of analytic capacity in terns of curvature of measures from \cite{Tolsa-sem} and the characterization of
the capacities $\Gamma_{d,n}$ in terms of $L^2$ estimates of Riesz transforms from \cite{Volberg} and \cite{Prat},
we will obtain Theorem \ref{teo1} and  Theorem \ref{teo2}
respectively.

\vv

\section{Notation and preliminaries} \label{secprelim}

\subsection{Generalities}
We write $a\lesssim b$ if there is a $C>0$ such that $a\leq Cb$, and we write $a\lesssim_{t} b$ if the constant $C$ depends on the parameter $t$. We write $a\approx b$ to mean $a\lesssim b\lesssim a$ and define $a\approx_{t}b$ similarly. 

We denote the open ball of radius $r$ centered at $x$ by $B(x,r)$. For a ball $B=B(x,r)$ and $\delta>0$ we write $r(B)$ for its radius and denote $\delta B=B(x,\delta r)$. 

Given an $m$-plane $V\subset \R^d$, $z\in\R^d$, and $s>0$, we consider the (open) cone
$$K(z,V,s) = \bigl\{x\in\R^d:\dist(x-z,V)<s\,|x-z|\bigr\}.$$
In the case $z=0$, we also write $K(V,s) = K(0,V,s)$.

\subsection{Measures and rectifiability}

The Lebesgue measure of a set $A\subset \R^{d}$ is denoted by $\LL^d(A)$. Given $0<\delta\leq\infty$, we set
\[\HH^{n}_{\delta}(A)=\inf\left\{\textstyle{ \sum_i \diam(A_i)^n: A_i\subset\R^{d},\,\diam(A_i)\leq\delta,\,A\subset \bigcup_i A_i}\right\}.\]
We define the {\it $n$-dimensional Hausdorff measure} as
\[\HH^{n}(A)=\lim_{\delta\to 0}\HH^{n}_{\delta}(A).\]

A set $E\subset \R^d$ is called {\em $n$-rectifiable} if there are Lipschitz maps
$f_i:\R^n\to\R^d$, $i=1,2,\ldots$, such that 
\begin{equation}\label{eq001}
\HH^n\biggl(E\setminus\bigcup_i f_i(\R^n)\biggr) = 0.
\end{equation}
On the other hand, $E$ is called {\em purely $n$-unrectifiable} if any $n$-rectifiable subset $F\subset E$ has zero $\HH^n$-measure.

Also, one says that 
a Radon measure $\mu$ on $\R^d$ is {\em $n$-rectifiable} if $\mu$ vanishes out of an $n$-rectifiable
set $E\subset\R^d$ and moreover $\mu$ is absolutely continuous with respect to $\HH^n|_E$.

A measure $\mu$ is called {\em $n$-AD-regular} (or just {\em AD-regular} or {\em Ahlfors-David regular}) if there exists some
constant $c>0$ such that
$$c^{-1}r^n\leq \mu(B(x,r))\leq c\,r^n\quad \mbox{ for all $x\in
\supp(\mu)$ and $0<r\leq \diam(\supp(\mu))$.}$$

\subsection{Cauchy and Riesz transform, and capacities}

Given a signed Radon measure $\nu$ in $\C$, the Cauchy transform is defined by
$$\CC\nu(x) = \int \frac1{z-w}\,d\nu(w),$$
whenever the integral makes sense. The $\ve$-truncated version is  
$$\CC_\ve \nu(x) = \int_{|z-w|>\ve} \frac1{z-w}\,d\nu(w).$$
For a signed Radon measure $\nu$ in $\R^d$, we consider the $n$-dimensional Riesz
transform
$$\RR^n\nu(x) = \int \frac{x-y}{|x-y|^{n+1}}\,d\nu(y),$$
whenever the integral makes sense. For $\ve>0$, its $\ve$-truncated version is given by 
$$\RR_\ve^n \nu(x) = \int_{|x-y|>\ve} \frac{x-y}{|x-y|^{n+1}}\,d\nu(y).$$

The curvature of a non-negative Borel measure $\mu$ in $\C$ is defined by
$$c^2(\mu) = \iiint \frac1{R(x,y,z)^2}\,d\mu(x)\,d\mu(y)\,d\mu(z),$$
where $R(x,y,z)$ stands for the radius of the circumference passing through $x,y,z$.
For $\ve>0$, its $\ve$-truncated version is
$$c^2_\ve(\mu) = \iiint_{\begin{subarray}{l}|x-y|>\ve\\|x-z|>\ve\\|y-z|>\ve\end{subarray}}
 \frac1{R(x,y,z)^2}\,d\mu(x)\,d\mu(y)\,d\mu(z).$$
As shown in \cite{MV}, if $\mu$ is a finite Borel measure in $\C$ satisfying the linear growth condition
$$\mu(B(z,r))\leq c_0\,r\quad \mbox{ for all $z\in\C$, $r>0$,}$$
then
$$\|\CC_\ve\mu\|_{L^2(\mu)}^2 = \frac16\,c^2_\ve(\mu) + O(\mu(\C)),$$
where $|O(\mu(\C))|\lesssim \mu(\C)$, with the implicit constant depending only on $c_0$.
The connection between the Cauchy kernel and curvature of measures was first observed by Melnikov while studying analytic
capacity \cite{Melnikov}.

We denote by $L_n(E)$ the set of positive Borel measures $\mu$ supported on $E$ satisfying $$\mu(B(x,r))\leq r^n
\quad \mbox{ for all $x\in E$, $r>0$.}$$ 
The following theorem characterizes analytic capacity in terms of measures from $L_1(E)$ with finite curvature.

\begin{theorem}\label{teogam}
Let $E\subset\C$ be compact. Then we have:
\begin{equation}\label{eqgam1*}
\gamma(E)\approx \sup\bigl\{\mu(E):\mu\in L_1(E),\,c^2(\mu)\leq \mu(E)\bigr\}.
\end{equation}
\end{theorem}

The fact that $\gamma(E)$ is bigger than a constant multiple of the supremum is due to Melnikov \cite{Melnikov}, and the
more difficult converse estimate to Tolsa \cite{Tolsa-sem}.

The extension of the preceding result to the capacities $\Gamma_{d,n}$ is the following:

\begin{theorem}\label{teokap}
Let $E\subset\R^d$ be compact. Then we have:
\begin{equation}\label{eqkap1*}
\Gamma_{d,n}(E)\approx \sup\Bigl\{\mu(E):\mu\in L_n(E),\,\sup_{\ve>0}\|\RR^n_\ve\mu\|_{L^2(\mu)}^2\leq \mu(E) \Bigr\}.
\end{equation}
\end{theorem}

Theorem \ref{teokap} was proved by Volberg in the case $n=d-1$, and it was later extended to all the values $0<n<d$ by Prat \cite{Prat}.

\vv

\section{The Fourier calculation}

\subsection{The planar case}

We think it is worth first studying the planar case because it is simpler the higher dimensional case, and the
result is more precise. 

Recall that for any Schwartz function $\phi : \R^2 \to \C$, we have
\begin{align}
\label{2d-fourier-transform-identity}
\wh{P_\theta\phi}(x) = \wh\phi(x) \quad\text{ for all $x \in L_\theta$},
\end{align}
where $(P_\theta\phi)(x) = \int_{x + L_\theta^\perp} \phi \, d\LL^1 = \int_{L_\theta^\perp} \phi(x + y) \, d\LL^1y$ and $\wh{P_\theta\phi}$ denotes the 1-dimensional Fourier transform on $L_\theta$. (Note that $P_\theta\phi$ is the density of $P_\theta\nu$ where $\nu = \phi(x) \, dx$.) To see \eqref{2d-fourier-transform-identity}, observe that for $\xi \in L_\theta$,
\[
\wh{P_\theta\phi}(\xi)
=
\int_{L_\theta}
P_\theta\phi(x)
e^{-2\pi i x \cdot \xi} \, dx
=
\int_{L_\theta}
\int_{L_\theta^\perp}
\phi(x+y)
e^{-2\pi i x \cdot \xi} \, dx
\, dy
=
\hat\phi(\xi),
\]
where we use the fact that $y \cdot \xi = 0$ for $y \in L_\theta^\perp$.

\begin{lemma}
\label{lem-2d-fourier-transform-of-kernel}
Let $K_I$ be the cone $K_I = \bigl\{re^{i\theta}:r\in\R,\,\theta\in I\bigr\}$. Let $I^\bot= I+\frac\pi2$ (mod $\pi$) and define the cone $K_{I^\bot}$ similarly. Then the (distributional) Fourier transform of $\chi_{K_I}(x)\, |x|^{-1}$ is $\chi_{K_{I^\bot}}(x)\, |x|^{-1}$.
\end{lemma}

\begin{proof}
Let
$\phi: \R^2 \to \C$ be a Schwartz function. By applying the identity \eqref{2d-fourier-transform-identity} to $\phi$ and using the Fourier inversion formula, we have $\int_{L_\theta} \wh \phi \, d\LL^1 =  (P_\theta\phi)(0) = \int_{L_\theta^\perp} \phi \, d\LL^1$. Thus, by polar coordinates,
\begin{align*}
\int_{K_I}
|x|^{-1} \,\wh \phi(x)
\, dx
=
 \int_I \int_{L_\theta} \wh \phi \, d\LL^1 \, d\theta
=
 \int_I \int_{L_\theta^\perp} \phi \, d\LL^1 \, d\theta
=
\int_{K_{I^\perp}}
|x|^{-1}\, \phi(x)
\, dx,
\end{align*}
which completes the proof.
\end{proof}

\begin{propo}\label{lemfourier}
Let $\psi : \R^2 \to \R$ be a Schwartz function. Then, for any set $I\subset [0,\pi]$, we have
$$\int_{I^\bot} \|P_\theta\psi\|_2^2\,d\theta = \iint_{x-y\in K_I} \frac{\psi(x)\, \psi(y)}{|x-y|}\,dx\,dy.$$
\end{propo}

\begin{proof}
Let $k(x) = \chi_{K_I}(x)\, |x|^{-1}$, so that $\wh k(x) = \chi_{K_{I^\perp}}(x) |x|^{-1}$ and $\iint_{x-y\in K_I} \frac{\psi(x)\, \psi(y)}{|x-y|}\,dx\,dy = \int (k * \psi) \, \psi \, dx$.  Since $\psi$ is a real-valued Schwartz function, we have
$\int (k * \psi) \, \psi \, dx
=
\int \wh k \,|\wh\psi|^2\,dx$,
so it follows that
\[
\iint_{x-y\in K_I} \frac{\psi(x)\, \psi(y)}{|x-y|}\,dx\,dy
=
\int (k * \psi) \, \psi \, dx
=
\int \wh k \,|\wh\psi|^2\,dx
=
\int_{K_{I^\perp}} |x|^{-1} \,|\wh\psi(x)|^2 \, dx
.
\]

Finally, by Plancherel and the identity \eqref{2d-fourier-transform-identity} applied to $\psi$, we have
\begin{equation}\label{eq843}
\int_{I^\perp} \|P_\theta\psi\|_2^2\,d\theta
=
\int_{I^\perp} \int_{L_\theta} |\widehat{P_\theta\psi}|^2 \, d\LL^1 \,d\theta
=
\int_{I^\perp} \int_{-\infty}^\infty |\wh\psi(r\,e^{i\theta})|^2 \, dr \,d\theta
=
 \int_{K_{I^\perp}} |x|^{-1} |\wh\psi(x)|^2 \, dx,
\end{equation}
which completes the proof.
\end{proof}

\vv

\begin{coro}\label{corofourier}
Let $\mu$ be a finite Borel measure in $\C$ with compact support and $I\subset [0,\pi]$ an arbitrary open set. Then we have
$$\iint_{x-y\in K_I\setminus\{0\}} \frac1{|x-y|}\,d\mu(x)\,d\mu(y)\leq \int_{I^\bot} \|P_\theta\mu\|_2^2\,d\theta,$$
where $K_I$ is the cone
$K_I = \bigl\{re^{i\theta}:r\in\R,\,\theta\in I\bigr\}$
and $I^\bot= I+\frac\pi2$ (mod $\pi$).
\end{coro}

\begin{proof}We assume that the integral on the right hand side above is finite.
Fix $\phi : \R^2 \to \R$ a $C^\infty$ radial bump function and for $\ve>0$, let $\phi_\ve(x) = \frac{1}{\ve^2} \phi(\frac{x}{\ve})$.
Denote $\mu_\ve = \mu * \phi_\ve$. It is straightforward to check that the identity \rf{eq843} holds both for $\mu$ and $\mu_\ve$.
Then, by the dominated convergence theorem we deduce that
\begin{align*}
\lim_{\ve\to0} \int_{I^\bot} \|P_\theta\mu_\ve\|_2^2\,d\theta &
=  \lim_{\ve\to0}\int_{K_{I^\perp}} |x|^{-1} |\wh\mu(x)\,\wh\phi(\ve x)|^2 \, dx\\
& 
=   \int_{K_{I^\perp}} |x|^{-1} |\wh\mu(x)|^2 \, dx =
\int_{I^\bot} \|P_\theta\mu\|_2^2\,d\theta.
\end{align*}
Hence, by Proposition \ref{lemfourier} we have
$$\limsup_{\ve\to0} \iint_{x-y\in K_I} \frac1{|x-y|}\,d\mu_\ve(x)\,d\mu_\ve(y)\leq \int_{I^\bot} \|P_\theta\mu\|_2^2\,d\theta.$$
So it suffices to show that
\begin{equation}\label{eq7327}
\iint_{x-y\in K_I\setminus\{0\}} \frac1{|x-y|}\,d\mu(x)\,d\mu(y)\leq\limsup_{\ve\to0} \iint_{x-y\in K_I} \frac1{|x-y|}\,d\mu_\ve(x)\,d\mu_\ve(y).
\end{equation}
To this end, consider an arbitrary non-negative continuous, compactly supported, function $f(x)\leq \chi_{K_I}(x)\, |x|^{-1}$.
We have that 
$$\int f * \mu_\ve \,d\mu_\ve = \int (f * \mu*\phi_\ve *\phi_\ve)\,d\mu.$$
Since $f*\mu$ is compactly supported and continuous, $f * \mu*\phi_\ve *\phi_\ve$ converges uniformly to $f*\mu$ as $\ve\to0$, and
thus
$$\int f*\mu\,d\mu = \lim_{\ve\to0}\int f * \mu_\ve \,d\mu_\ve \leq\limsup_{\ve\to0} \iint_{x-y\in K_I} \frac1{|x-y|}\,d\mu_\ve(x)\,d\mu_\ve(y) .$$
As this holds uniformly for any continuous compactly supported function $f$ such that $0\leq f\leq \chi_{K_I}(x)\, |x|^{-1}$ and $K_I\setminus\{0\}$ is open,
\rf{eq7327} follows by the monotone convergence theorem. 
\end{proof}

\vv

\subsection{The higher dimensional case}\label{sechighdim}

Let $G(d, n)$ denote the Grassmanian space of $n$-dimensional linear subspaces of $\R^d$. Let $\gamma_{d,n}$ denote the natural probability measure on $G(d, n)$. For $V \in G(d,n)$, let $P_V : \R^d \to V$ be the orthogonal projection onto $V$. We define a metric on $G(d, n)$ by $d(V, W) = \| P_V - P_W \|$, where $\| \cdot \|$ denotes the operator norm. 

Recall that for any Schwartz function $\phi : \R^d \to \C$ and any $V \in G(d, n)$, we have
\begin{align}
\label{dd-fourier-transform-identity}
\wh{P_V\phi}(x) = \wh\phi(x) \quad\text{ for all $x \in V$},
\end{align}
where $(P_V\phi)(x) = \int_{x + V^\perp} \phi \, d\LL^{d-n}$ and $\wh{P_V\phi}$ denotes the $n$-dimensional Fourier transform on $V$.  The proof is identical to the proof of \eqref{2d-fourier-transform-identity}.

The following is the higher dimensional analogue of Lemma \ref{lem-2d-fourier-transform-of-kernel}.

\begin{lemma}
Let $B \subset G(d, n)$. Let $\sigma, \nu$ be measures on $\R^d$ given by
\begin{align}
\label{eq:defn-mu}
\int f \, d\sigma
&=
\int_B \int_V f \, d\LL^n \, d\gamma_{d,n}(V),
\\
\int f \, d\nu
&=
\int_B \int_{V^\perp} f \, d\LL^{d-n} \, d\gamma_{d,n}(V).
\end{align}
Then the (distributional) Fourier transform of $\sigma$ is $\nu$.
\end{lemma}

\begin{proof}
Let $\phi : \R^d \to \C$ be a Schwartz function. Then, as in the proof of Lemma \ref{lem-2d-fourier-transform-of-kernel},
\begin{align*}
\int \wh\phi \, d\sigma
&=
\int_B \int_V \wh \phi \, d\LL^n \, d\gamma_{d,n}(V)
=
\int_B \int_V \widehat{P_V\phi} \, d\LL^n \, d\gamma_{d,n}(V)
\\
&=
\int_B P_V\phi(0) \, d\gamma_{d,n}(V)
=
\int_B \int_{V^\perp} \phi \, d\LL^{d-n} \, d\gamma_{d,n}(V)
=
\int \phi \, d\nu.
\qedhere
\end{align*}
\end{proof}

\vv
\begin{lemma}\label{eq:density-upper-bound}
Let $B \subset G(d, n)$. Let $\sigma$ be given by \eqref{eq:defn-mu}. Then $\supp \sigma \subset \bigcup_{V \in \bar B} V$, $\sigma \ll \LL^d$, and $\frac{d\sigma}{dx}(x) \leq \frac{c(d,n)}{|x|^{d-n}}$. 
\end{lemma}

\begin{proof}
From the definition of $\sigma$, it is immediate that $\supp \sigma \subset \bigcup_{V \in \bar B} V$. The next two properties follow from the following identity:
\[
\int_{G(d,n)}
\int_V f \, d\LL^n \, d\gamma_{d,n}(V)
=
c(d,n)
\int_{\R^d} \frac{f(x)}{|x|^{d-n}}\, d\LL^d.
\]
For a proof of this identity, see (24.2) in \cite{Mattila-Fourier}.
\end{proof}

\begin{lemma}\label{lemfourier-higher-dim}
Let $\psi : \R^2 \to \R$ be a Schwartz function. Then for any set $B \subset G(d, n)$, we have
\[
\int_{B} \|P_V\psi\|_2^2\,d\gamma_{d,n}(V)
=
\iint
\frac{d\nu}{dx}(x-y)
\, \psi(x) \, \psi(y) \, dx \, dy.
\]
\end{lemma}

\begin{proof}
The proof is identical to the proof of Proposition \ref{lemfourier}.
\end{proof}

To make Lemma \ref{lemfourier-higher-dim} more useful, we obtain a lower bound on $\frac{d\nu}{dx}$ via the following three lemmas.

\begin{lemma}\label{lemma:grassmanian-ball}
For all $V \in G(d, n)$ and for $\delta \lesssim_{d,n} 1$,
\[
\gamma_{d,n}(B(V, \delta)) \approx_{d,n} \delta^{n(d-n)}
\]
\end{lemma}

\begin{proof}
This is Proposition 4.1 of \cite{Fassler-Orponen}.
\end{proof}

\begin{lemma}\label{lemma:grassmanian-isomorphic}
Let $x \in \R^d \setminus\{0\}$. Then $G_x := \{V \in G(d,n) : x \in V \}$ and $G(d-1, n-1)$ are isomorphic as metric spaces.
\end{lemma}

\begin{proof}
Let $A : \R^{d-1} \to \R^d$ be a linear map satisfying $A^TA = {\rm id}$ and whose image is the orthogonal complement of $x$. Consider the map $\Psi : G(d-1, n-1) \to G_x$ given by $V  \mapsto \operatorname{span}(AV, x)$. We will show $\Psi$ is an isometry. First we make two observations. 
\begin{enumerate}
\item For any $V \in G_x$, we have $P_V x = x$.
\item For any $z \in \R^{d-1}$ and $V \in G(d-1,n-1)$, we have $P_{\Psi V} Az = AP_Vz$.
\end{enumerate}

Let $V, W \in G(d-1,n-1)$. We need to show 
\[
\|P_{\Psi V} - P_{\Psi W} \|_{\R^d \to \R^d} = \| P_V - P_W \|_{\R^{d-1} \to \R^{d-1}}.
\]  
Let $y \in \R^d$ and write $y = \lambda x + Az$, where $\lambda \in \R$ and $z \in \R^{d-1}$. Note that $z = A^Ty$. Using the two observations above, we have
\[
(P_{\Psi V} - P_{\Psi W})y =   (P_{\Psi V} - P_{\Psi W})Az = A(P_{V} - P_{W})z =A(P_{V} - P_{W})A^Ty.
\]
Hence $P_{\Psi V} - P_{\Psi W} = A(P_{V} - P_{W})A^T$. Since $\|A\|_{\R^{d-1} \to \R^d} = \|A^T\|_{\R^d \to \R^{d-1}} = 1$, it follows that $\| P_{\Psi V} - P_{\Psi W} \| \leq \| P_{V} - P_{W} \|$. To show the reverse inequality, note that for any $z \in \R^{d-1}$, we have
\begin{align*}
|(P_V-P_W)z| 
&= 
|A(P_V-P_W)z| 
\\
&= 
|(P_{\Psi V}-P_{\Psi W})Az| 
\\
&\leq \| P_{\Psi V}-P_{\Psi W} \| |Az| 
\\
&= 
\| P_{\Psi V}-P_{\Psi W} \| |z|,
\end{align*}
which implies $\| P_{V} - P_{W} \| \leq \| P_{\Psi V} - P_{\Psi W} \|$.
\end{proof}

\begin{lemma}\label{eq:density-lower-bound}
Let $B = B(V_0, r) \subset G(d, n)$. Let $\sigma$ be given by \eqref{eq:defn-mu}. Then there is a $c = c(d, n, r) > 0$ such that 
\[
\frac{d\sigma}{dx}(x) \geq  \frac{c}{|x|^{d-n}}
\quad\text{ on the cone } 
\bigcup_{V \in B(V_0, \frac{1}{2}r)} V.
\]
\end{lemma}

\begin{proof}%
Note that $\sigma(\lambda A) = \lambda^n \sigma(A)$, which implies $\frac{d\sigma}{dx} ( \lambda x) = \lambda^{n-d} \frac{d\sigma}{dx}(x)$. Hence, it suffices to show $\frac{d\sigma}{dx}(x) \gtrsim_{d,n,r} 1$ for all $x \in \bigcup_{V \in B(V_0, \frac{1}{2}r)} V$ with $|x| = 1$. 

Fix $x \in \bigcup_{V \in B(V_0, \frac{1}{2}r)} V$ with $|x| = 1$. %
Let $G_x = \{V \in G(d,n) : x \in V \}$, and let $(G_x)^\delta \subset G(d, n)$ denote the $\delta$-neighborhood of $G_x$.

We claim that
\begin{align}\label{eq:bound-ball-by-nbhd}
\sigma(B(x,s))
\gtrsim
s^n\gamma_{d,n}((G_x)^{s/2} \cap B(V_0, r))
\qquad\text{for all } s > 0.
\end{align}
To see this, note that if $V \in (G_x)^{s/2}$, then there is some $W \in G_x$ such that $d(V, W) < \frac{s}{2}$. Then $|x - P_Vx| = |P_Wx - P_Vx| \leq \| P_W - P_V \| \leq d(W,V) < \frac{s}{2}$, so $\LL^n(V \cap B(x,s)) \gtrsim_{d,n} s^n$. Hence,
\begin{align*}
\sigma(B(x,s))
&=
\int_{B(V_0, r)}
\LL^n(V \cap B(x,s))
\, d\gamma_{d,n}(V)
\\
&\gtrsim_{d,n}
s^n\,\gamma_{d,n}((G_x)^{s/2} \cap B(V_0, r)),
\end{align*}
which proves \eqref{eq:bound-ball-by-nbhd}.

Next, we bound $\gamma_{d,n}((G_x)^{s/2} \cap B(V_0, r))$ from below. Fix a $V_1 \in G_x \cap B(V_0, \frac{1}{2}r)$. (Since $x \in \bigcup_{V \in B(V_0, \frac{1}{2}r)} V$, $V_1$ exists.) 

Suppose $s < r$. Let $F_s$ be a maximal $s$-separated subset of $G_x \cap B(V_1, \frac{1}{2}r)$. It follows from the maximality of $F_s$ that 
\begin{align}
\label{eqref:grassmanian-balls-cover}
G_x \cap B(V_1, \tfrac{1}{2}r) \subset \bigcup_{W \in F_s} (G_x \cap B(W,s)).
\end{align} 
By \eqref{eqref:grassmanian-balls-cover}, Lemma \ref{lemma:grassmanian-isomorphic} and Lemma \ref{lemma:grassmanian-ball} applied to $G(d-1, n-1)$, it follows that for $s \lesssim_{d,n,r} 1$,
\begin{align}\label{eq:card-F}
\# F_s \gtrsim_{d,n,r} s^{-(n-1)(d-n)}. 
\end{align}

Next, observe that the balls $\{B(W, \frac{s}{2})\}_{W \in F_s}$ are pairwise disjoint and contained in $(G_x)^{s/2} \cap B(V_0, r)$, so 
\begin{align}
\label{eq:bound-gamma-below}
\gamma_{d,n}((G_x)^{s/2} \cap B(V_0, r))
\geq
\sum_{W \in F_s} \gamma_{d,n}(B(W,\tfrac{s}{2}))
\gtrsim_{d,n,r}
s^{d-n}
\quad\text{ for } s \lesssim_{d,n,r} 1,
\end{align}
where we used \eqref{eq:card-F} and Lemma \ref{lemma:grassmanian-ball} in the last inequality. Finally, \eqref{eq:bound-ball-by-nbhd} and \eqref{eq:bound-gamma-below} imply $\frac{d\sigma}{dx}(x) \gtrsim_{d,n,r} 1$, as desired.
\end{proof}

\vv
\begin{coro}\label{corofourier1}
Let $V_0 \in G(d, n)$ and $s > 0$. Then there exist constants $\lambda,c>1$ such that for 
any Schwartz function $\psi : \R^d \to \R$,
\begin{multline*}
c^{-1}\! \iint_{x-y \in K(V_0^\bot,\lambda^{-1}s)}\frac{\psi(x)\, \psi(y)}{|x-y|^n}\,dx\,dy
 \leq 
\int_{B(V_0, s)} \| P_V\psi \|_2^2 \, d \gamma_{d,n}(V)\\  \leq c\iint_{x-y \in K(V_0^\bot,\lambda s)}
\frac{\psi(x)\, \psi(y)}{|x-y|^n}\,dx\,dy.
\end{multline*}
\end{coro}
\vv

\begin{proof}
By Lemma \ref{eq:density-upper-bound} and Lemma \ref{eq:density-lower-bound} (applied to $G(d, d-n)$),
\[
c_1 \frac{\chi_{K_1}(x)}{|x|^n} \leq \frac{d\nu}{dx}(x) \leq c_2 \frac{\chi_{K_2}(x)}{|x|^n},
\]
where $K_1 = \bigcup_{V \in B(V_0, \frac{1}{2}s)} V^\perp, K_2 = \bigcup_{V \in B(V_0, s)} V^\perp$.
 
For $\lambda$ sufficiently large, we have $K(V_0^\perp, \lambda^{-1}s) \subset K_1 \subset K_2 \subset K(V_0^\perp, \lambda s)$. Thus, 
\[
c_1 \frac{\chi_{K(V_0^\perp, \lambda^{-1}s)}(x)}{|x|^n} \leq \frac{d\nu}{dx}(x) \leq c_2 \frac{\chi_{K(V_0^\perp, \lambda s)}(x)}{|x|^n},
\]
so this corollary follows easily from Lemma \ref{lemfourier-higher-dim}.
\end{proof}

\vv
\begin{coro}\label{corofourier2}
Let $V_0 \in G(d, n)$ and $s > 0$. Then there exist constants $\lambda,c>1$ such that for 
any finite Borel measure $\mu$ in
$\R^d$,
$$
 \iint_{x-y \in K(V_0^\bot,s)}\frac{ d\mu (x)\,  d\mu (y)}{|x-y|^n} 
\leq c
\int_{B(V_0, \lambda s)} \| P_V\mu \|_2^2 \, d \gamma_{d,n}(V) .
$$
\end{coro}

The proof of this corollary follows from Corollary \ref{corofourier1}, along the same lines as the one of Corollary \ref{corofourier}, and so we skip it.

\vv

\begin{remark}
The converse inequality 
\begin{equation}\label{eq:converse-inequality}
\int_{B(V_0, s)} \| P_V\mu \|_2^2 \, d \gamma_{d,n}(V) \leq c
 \iint_{x-y \in K(V_0^\bot,\lambda s)} \frac{ d\mu (x)\,  d\mu (y)}{|x-y|^n}
\end{equation}
does not hold for arbitrary measures. Indeed, in the case $d=2$, $n=1$, consider a segment $L$ through the origin, and let $K(V_0^\bot,\lambda s)$
be a cone such that $L$ is not contained in the closure of the cone. Then with $\mu = \HH^1|_{L}$, the integral on the left hand side is positive (and finite) while the one
on the right hand side is zero.

However, if we modify \eqref{eq:converse-inequality} by adding an additional term to the right-hand side, we can make the inequality true. This result not needed in our paper, but we include the details in Appendix \ref{section:reverse-ineq}.
\end{remark}

\vv

\section{The dyadic lattice of David and Mattila}\label{sec:DMlatt}

Now we will introduce the dyadic lattice of cubes
with small boundaries of David-Mattila associated with a Radon measure $\mu$. This lattice has been constructed in \cite[Theorem 3.2]{David-Mattila}. 
Its properties are summarized in the next lemma.

\begin{lemma}[David, Mattila]
\label{lemcubs}
Let $\mu$ be a compactly supported Radon measure in $\R^{d}$.
Consider two constants $C_0>1$ and $A_0>5000\,C_0$ and denote $W=\supp\mu$. 
Then there exists a sequence of partitions of $W$ into
Borel subsets $Q$, $Q\in \DD_{\mu,k}$, with the following properties:
\begin{itemize}
\item For each integer $k\geq0$, $W$ is the disjoint union of the ``cubes'' $Q$, $Q\in\DD_{\mu,k}$, and
if $k<l$, $Q\in\DD_{\mu,l}$, and $R\in\DD_{\mu,k}$, then either $Q\cap R=\varnothing$ or else $Q\subset R$.
\vv

\item The general position of the cubes $Q$ can be described as follows. For each $k\geq0$ and each cube $Q\in\DD_{\mu,k}$, there is a ball $B(Q)=B(x_Q,r(Q))$ such that
$$x_Q\in W, \qquad A_0^{-k}\leq r(Q)\leq C_0\,A_0^{-k},$$
$$W\cap B(Q)\subset Q\subset W\cap 28\,B(Q)=W \cap B(x_Q,28r(Q)),$$
and
$$\mbox{the balls\, $5B(Q)$, $Q\in\DD_{\mu,k}$, are disjoint.}$$

\vv
\item The cubes $Q\in\DD_{\mu,k}$ have small boundaries. That is, for each $Q\in\DD_{\mu,k}$ and each
integer $l\geq0$, set
$$N_l^{ext}(Q)= \{x\in W\setminus Q:\,\dist(x,Q)< A_0^{-k-l}\},$$
$$N_l^{int}(Q)= \{x\in Q:\,\dist(x,W\setminus Q)< A_0^{-k-l}\},$$
and
$$N_l(Q)= N_l^{ext}(Q) \cup N_l^{int}(Q).$$
Then
\begin{equation}\label{eqsmb2}
\mu(N_l(Q))\leq (C^{-1}C_0^{-3d-1}A_0)^{-l}\,\mu(90B(Q)).
\end{equation}
\vv

\item Denote by $\DD_{\mu,k}^{db}$ the family of cubes $Q\in\DD_{\mu,k}$ for which
\begin{equation}\label{eqdob22}
\mu(100B(Q))\leq C_0\,\mu(B(Q)).
\end{equation}
We have that $r(Q)=A_0^{-k}$ when $Q\in\DD_{\mu,k}\setminus \DD_{\mu,k}^{db}$
and
\begin{equation}\label{eqdob23}
\mu(100B(Q))\leq C_0^{-l}\,\mu(100^{l+1}B(Q))\quad
\mbox{for all $l\geq1$ with $100^l\leq C_0$ and $Q\in\DD_{\mu,k}\setminus \DD_{\mu,k}^{db}$.}
\end{equation}
\end{itemize}
\end{lemma}

\vv

We use the notation $\DD_\mu=\bigcup_{k\geq0}\DD_{\mu,k}$. Observe that the families $\DD_{\mu,k}$ are only defined for $k\geq0$. So the diameter of the cubes from $\DD_\mu$ are uniformly
bounded from above.
For $Q\in\DD_\mu$, we set $\DD_\mu(Q) =\{P\in\DD_\mu:P\subset Q\}$.
Given $Q\in\DD_{\mu,k}$, we denote $J(Q)=k$, and we set
$\ell(Q)= 56\,C_0\,A_0^{-k}$ and we call it the side length of $Q$. Notice that 
$$C_0^{-1}\ell(Q)\leq \diam(28B(Q))\leq\ell(Q).$$
Observe that $r(Q)\approx\diam(Q)\approx\ell(Q)$.
Also we call $x_Q$ the center of $Q$, and the cube $Q'\in \DD_{\mu,k-1}$ such that $Q'\supset Q$ the parent of $Q$.
 We set
$B_Q=28 B(Q)=B(x_Q,28\,r(Q))$, so that 
$$W\cap \tfrac1{28}B_Q\subset Q\subset B_Q.$$

We assume $A_0$ big enough so that the constant $C^{-1}C_0^{-3d-1}A_0$ in 
\rf{eqsmb2} satisfies 
$$C^{-1}C_0^{-3d-1}A_0>A_0^{1/2}>10.$$
Then we deduce that, for all $0<\lambda\leq1$,
\begin{align}\label{eqfk490}\nonumber
\mu\bigl(\{x\in Q:\dist(x,W\setminus Q)\leq \lambda\,\ell(Q)\}\bigr) + 
\mu\bigl(\bigl\{x\in 3.5B_Q\setminus Q:\dist&(x,Q)\leq \lambda\,\ell(Q)\}\bigr)\\
&\leq
c\,\lambda^{1/2}\,\mu(3.5B_Q).
\end{align}

We denote
$\DD_\mu^{db}=\bigcup_{k\geq0}\DD_{\mu,k}^{db}$.
Note that, in particular, from \rf{eqdob22} it follows that
\begin{equation}\label{eqdob*}
\mu(3B_{Q})\leq \mu(100B(Q))\leq C_0\,\mu(Q)\qquad\mbox{if $Q\in\DD_\mu^{db}.$}
\end{equation}
For this reason we will call the cubes from $\DD_\mu^{db}$ doubling. 
Given $Q\in\DD_\mu$, we set $\DD_\mu^{db}(Q) = \DD^{db}_\mu\cap\DD_\mu(Q)$.

As shown in \cite[Lemma 5.28]{David-Mattila}, every cube $R\in\DD_\mu$ can be covered $\mu$-a.e.\
by a family of doubling cubes:
\vv

\begin{lemma}\label{lemcobdob}
Let $R\in\DD_\mu$. Suppose that the constants $A_0$ and $C_0$ in Lemma \ref{lemcubs} are
chosen suitably. Then there exists a family of
doubling cubes $\{Q_i\}_{i\in I}\subset \DD_\mu^{db}$, with
$Q_i\subset R$ for all $i$, such that their union covers $\mu$-almost all $R$.
\end{lemma}

The following result is proved in \cite[Lemma 5.31]{David-Mattila}.
\vv

\begin{lemma}\label{lemcad22}
Suppose that the constants $A_0$ and $C_0$ in Lemma \ref{lemcubs} are
chosen suitably.
Let $R\in\DD_\mu$ and let $Q\subset R$ be a cube such that all the intermediate cubes $S$,
$Q\subsetneq S\subsetneq R$ are non-doubling (i.e.\ belong to $\DD_\mu\setminus \DD_\mu^{db}$).
Then%
\begin{equation}\label{eqdk88}
\mu(100B(Q))\leq A_0^{-10d(J(Q)-J(R)-1)}\mu(100B(R)).
\end{equation}
\end{lemma}

We remark that for the preceding two lemmas to hold, we need to choose $A_0$ much larger than $C_0$. From now on we assume 
this condition to hold.

Given a ball (or an arbitrary set) $B\subset \R^{d}$ and a fixed $n\geq 1$, we consider its $n$-dimensional density:
$$\Theta_\mu(B)= \frac{\mu(B)}{\diam(B)^n}.$$

From the preceding lemma we deduce:

\vv
\begin{lemma}\label{lemcad23}
Let $Q,R\in\DD_\mu$ be as in Lemma \ref{lemcad22}.
Then
$$\Theta_\mu(100B(Q))\leq (C_0A_0)^d\,A_0^{-9d(J(Q)-J(R)-1)}\,\Theta_\mu(100B(R))$$
and
$$\sum_{S\in\DD_\mu:Q\subset S\subset R}\Theta_\mu(100B(S))\leq c\,\Theta_\mu(100B(R)),$$
with $c$ depending on $C_0$ and $A_0$.
\end{lemma}

For the easy proof, see
 \cite[Lemma 4.4]{Tolsa-memo}, for example.

\vv
We will also need the following technical result.

\begin{lemma}\label{lemcubdob*}
Let $R\in\DD_\mu$ such that $\mu(2B_R)\leq C_1\,\mu(R)$. Then there exists another cube $Q\subsetneq R$ from 
$\DD_\mu^{db}$ such that
$$\mu(Q)\approx\mu(R) \quad \mbox{ and } \quad \ell(Q)\approx\ell(R),$$
with the implicit constants depending on $C_1$.
\end{lemma}

\begin{proof}
Suppose that $R\in\DD_{\mu,k}$. For some $N>1$ to be fixed later, denote by $I_N$ the family cubes from 
$\DD_{\mu,k+N}$ which are contained in $R$.
Recall that the balls $B(Q)$, $Q\in I_N$, are disjoint and that their radii satisfy
$$A_0^{-k-N}\leq r(Q)\leq C_0\,A_0^{-k-N}.$$
All the balls from $I_N$ intersect $R$ and are contained in $2B_R$ for $N$ big enough.
Thus we have
$$\#I_N\cdot C_d\, A_0^{(-k-N)d}\leq \LL^d\Bigl(\bigcup_{Q\in I_N} B(Q)\Bigr) \leq C_d\,\bigl(2\,C_0\,A_0^{-k}\bigr)^d,$$
and so
\[
\#I_N \leq 2^d\,C_0^d\,A_0^{Nd}.
\]
Therefore, the cube $Q'\in I_N$ with maximal measure satisfies
\begin{equation}\label{eqcontra*n}
\mu(Q') \geq 2^{-d}\,C_0^{-d}\,A_0^{-Nd}\mu(R).
\end{equation}
We claim now that if $N$ is big enough then there exists some cube $Q\in\DD_\mu^{db}$ such that
$Q'\subset Q\subsetneq R$. %
 Indeed, if such cube $Q$ does not exist, then denoting by $R'$ the son of $R$ that contains $Q'$,
  we deduce
$$\mu(Q')\leq \mu(100B(Q'))\leq A_0^{-10d(N-2)}\mu(100B(R'))\leq A_0^{-10d(N-2)}\mu(2B_R) \leq C_1\,
A_0^{-10d(N-2)}\mu(R),$$
taking into account that $100B(R')\subset2B_R$. For $N$ big enough, this estimate contradicts \rf{eqcontra*n}, and thus the cube $Q$ mentioned above exists. It is clear that this satisfies the
estimates $\mu(Q)\approx\mu(R)$ and $\ell(Q)\approx\ell(R)$, as wished.
\end{proof}

\vv

\vv

\section{The corona decomposition}\label{section:corona}

Let $\mu$ a Borel measure in $\R^d$ satisfying the growth condition
\begin{equation}\label{eqgrowth00}
\mu(B(x,r))\leq c_0\,r^n\quad\mbox{ for all $x\in\R^d$, $r>0$.}
\end{equation}
We consider the dyadic lattice $\DD_\mu$ of David-Mattila associated to $\mu$, and we assume that $\supp\mu\in\DD_\mu$ is the biggest cube in this lattice. (To this end, we assume $\DD_{\mu,k}$ to be defined for $k\geq k_0$, with an appropriate $k_0$.) Sometimes we will also denote by $R_0$ the initial
cube $\supp\mu$.
We allow all constants $c$, $C$, and other implicit constants to depend on $n$, $d$, and the parameters in the definition of the David-Mattila cubes.

Let $\ttt$ be a family of cubes from $\DD_\mu^{db}$ to be fixed below, with $R_0\in\ttt$.
For every $R\in\ttt$, denote by $\nex(R)$ the family of maximal cubes $Q\in\ttt$ that are contained in $R$, and by $\tr(R)$ the family of
cubes $Q\in\DD_\mu$ that are contained in $R$ and not contained in any $Q'\in\nex(R)$. Then, define 
$$\good(R) = R\setminus \bigcup_{Q\in\nex(R)} Q,$$
and for $Q,S\in\DD_\mu$ with $Q\subset S$,
$$\delta_\mu(Q,S) = \int_{2B_S\setminus 2B_Q}\frac1{|y-x_Q|^n}\,d\mu(y).$$

The next lemma is the main tool which will allow us to connect 
the energy $$\iint_{x-y\in K} \frac1{|x-y|^n}\,d\mu(x)\,d\mu(y)$$ to the curvature of $\mu$.

\begin{lemma}[Corona decomposition]\label{lemcorona}
Given $V_0\in G(d,d-n)$ and $s>0$, consider the cone $K:=K(V_0,s)\subset\R^d$.
Let $\mu$ be a Borel measure in $\R^d$ satisfying the growth condition \rf{eqgrowth00}.
 There exists a family $\ttt\subset\DD_\mu^{db}$ as above such that, for all $R\in\ttt$, there exists an $n$-dimensional 
 Lipschitz graph $\Gamma_R$ with the slope depending only on 
$s$ such that:
\begin{itemize}
\item[(a)] $\mu$-almost all $\good(R)$ is contained in $\Gamma_R$.

\item[(b)] For all $Q\in\nex(R)$ there exists another cube $\wt Q\in\DD_\mu$ such that $\delta_\mu(Q,\wt Q)\leq c\,\Theta_\mu(2B_R)$ and $2B_{\wt Q}\cap\Gamma_R\neq\varnothing$.

\item[(c)] For all $Q\in\tr(R)$, $\Theta_\mu(2B_Q)\leq c\,\Theta_\mu(2B_R)$.
\end{itemize}
Furthermore, the following packing condition holds:
\begin{equation}\label{eqpack*1} 
\sum_{R\in\ttt}\Theta_\mu(2B_R)\,\mu(R)\lesssim \mu(R_0) + \iint_{x-y\in K} \frac1{|x-y|^n}\,d\mu(x)\,d\mu(y),
\end{equation}
with the implicit constant depending only on $c_0$ and the aperture $s$ of the cone $K$.
\end{lemma}

The next Sections \ref{sec5}-\ref{sec8} are devoted to the proof of this lemma.
In these sections we assume that $\mu$ is a measure in $\R^d$ that satisfies \rf{eqgrowth00} and that $K=K(V_0,s)$ is a cone,
with $V_0\in G(d,d-n)$, $s>0$, 
such that
$$\iint_{x-y\in K} \frac1{|x-y|^n}\,d\mu(x)\,d\mu(y)<\infty.$$

\vv

\section{The construction of an approximate Lipschitz graph}\label{sec5}

From now on we will allow all the constants denoted by $c$ or $C$, and all the implicit constants in the relations $\lesssim$
and $\approx$ to depend on the parameters $C_0$ and $A_0$ of the David-Mattila lattice.

\subsection{The stopping cubes}

We consider constants $A\gg1$, $0<\ve \ll \tau \ll 1$, and $0 < \eta \ll1$ to be fixed below.
For $Q\in\DD_\mu$, we denote
\begin{equation}
\label{eq:defn-EEmuQ}
\EE_\mu(Q) = \frac1{\mu(Q)}\iint_{\begin{subarray}{l}x\in 2B_Q\\x-y\in K\\ \eta\,\ell(Q)\leq |x-y|\leq \eta^{-1}\ell(Q)
\end{subarray}} \frac1{|x-y|^n}\,d\mu(x)\,d\mu(y).\end{equation}
Observe that $\EE_\mu(Q)$ ``scales'' like $\Theta_\mu(2B_Q)$. (That is, both quantities have the same ``physical dimensions'' -- two factors of $\mu$ in the numerator and one factor of length in the denominator.)

Given a cube $R\in\DD_\mu^{db}$, we consider the following families of cubes:
\begin{itemize}
\item The high density family $\HD_0(R)$, which is made up of the cubes $Q\in\DD_\mu^{db}(R)$ which satisfy
$\Theta_\mu(2B_Q)\geq A\, \Theta_\mu(2B_R)$. 

\item The low density family $\LD_0(R)$, which is made up of the cubes $Q\in\DD_\mu(R)$ which satisfy
$\Theta_\mu(2B_Q)\leq \tau\, \Theta_\mu(2B_R)$.

\item The family $\BCE_0(R)$ of cubes with big conical Riesz energy, which is made up of the cubes $Q\in\DD_\mu(R) \setminus (\HD_0(R)\cup\LD_0(R))
$ such that
$$\sum_{S\in\DD_\mu:Q\subset S \subset R} \EE_\mu(S)
\geq \ve\,\Theta_\mu(2B_R).
$$
\end{itemize}

We denote by $\sss(R)$ the maximal (and thus disjoint) subfamily from $\HD_0(R)\cup\LD_0(R)
\cup\BCE_0(R)$, and we set
\begin{align*}
\HD(R) = \HD_0(R)\cap\sss(R),&\quad \LD(R) = \LD_0(R)\cap\sss(R),
\quad
\BCE(R) = \BCE_0(R)\cap\sss(R).
\end{align*}
Notice that the cubes from $\HD(R)$ are doubling, while the cubes from $\LD(R)$ and $\BCE(R)$ may be non-doubling.

We let $\tree(R)$ denote the subfamily of the cubes from $\DD_\mu(R)$ which are not strictly contained in any cube
from $\sss(R)$. (Note that it is possible for $\tree(R)$ to contain only the cube $R$ itself.)

\subsection{Preliminary estimates}
In this subsection we assume that $R\in\DD_\mu^{db}$.
The following statement is an immediate consequence of the construction of $\sss(R)$ and $\tree(R)$.

\begin{lemma}\label{lemdobbb1}
If $Q\in\DD_\mu$ and $Q\in \tree(R)\backslash \sss(R)$, then 
$$\tau\,\Theta_\mu(2B_R)\leq \Theta_\mu(2B_Q) \leq c\,A\,\Theta_\mu(2B_R).$$
Further, the second inequality also holds if $Q\in\sss(R)$.
\end{lemma}

\begin{proof}
The fact that $\Theta_\mu(2B_Q)\geq \tau\,\Theta_\mu(2B_R)$ for all $Q\in \tree(R)\backslash \sss(R)$ follows from 
the definition of the family $\LD(R)$.
To check that for such cubes it also holds $\Theta_\mu(2B_Q) \leq c\,A\,\Theta_\mu(2B_R)$, note first that this holds if $Q\in\DD_\mu^{db}(R)$ (with $c=1$).
If $Q\not\in\DD_\mu^{db}(R)$, let $P(Q)\in\DD_\mu^{db}(R)$ be the smallest doubling cube that contains $Q$ (such a cube exists because $R\in
\DD_\mu^{db}(R)$), so that $\Theta_\mu(2B_{P(Q)}) \leq A\,\Theta_\mu(2B_R)$.
For $j\geq 0$, denote by $Q_j$ the $j$-th ancestor of $Q$ (i.e.\ $Q_j\in\DD_\mu$ is such that $Q\subset Q_j$ and $\ell(Q_j)=A_0^j\,\ell(Q)$).
Let $i\geq 0$ be such that $P(Q)= Q_i$.
Since the cubes $Q_1,\ldots,Q_{i-1}$ do not belong to $\DD^{db}_\mu$, by Lemma \ref{lemcad23} we have
$$\Theta_\mu(2B_Q)\lesssim\Theta_\mu(100B(Q))\leq c(C_0,A_0)\,A_0^{-9di}\,\Theta_\mu(100B(Q_i))
\approx A_0^{-9di}\,\Theta_\mu(2B_{Q_i}) \lesssim A\,\Theta_\mu(2B_R).$$

Finally, if $Q\in\sss(R)$, we just need to take into account that $\Theta_\mu(2B_Q)\lesssim \Theta_\mu(2B_{\wh Q})$, 
where $\wh Q$ is the parent of $Q$.
\end{proof}
\vv

\vv

The next result is essentially proven in Lemma 6.3 from \cite{Azzam-Tolsa-GAFA}. %

\begin{lemma}\label{lemdob32}
There is some constant $C(A,\tau)>0$ big enough so that for any cube 
$Q\in\tree(R)$  there exists some cube $Q'\supset Q$ such that $Q'\in\DD_\mu^{db}\cap \tree(R)$
and $\ell(Q')\leq C(A,\tau)\,\ell(Q)$. 
If $R\not\in\sss(R)$, then we can take $Q'\in\tree(R)\setminus\sss(R)$.
\end{lemma}

\vv

\subsection{A key estimate}

For $x\in\R^d$ and $\lambda>0$ we denote
$$K_x^\lambda = K(x,V_0,\lambda s),$$
Given $Q\in\DD_\mu$, we also set
$$K_Q^\lambda = \bigcup_{x\in 2B_Q} K_x^\lambda.$$
In the case $\lambda=1$, we write $K_x$ and $K_Q$ instead of $K_x^1$ and $K_Q^1$.
\vv

\begin{lemma}\label{keylemma}
There exists some constant $M>1$ depending only on $s$ such that the following are true.
\begin{enumerate}
\item[(a)]
Suppose $Q \in \tree(R)$ and $Q' \in \DD_\mu(R)$ satisfy $Q' \cap (K_Q^{1/2} \setminus MB_Q) \neq \varnothing$ and $\ell(Q') \leq \frac{1}{M} \dist(Q', Q)$. Then $Q' \not\in \tree(R)$.
\item[(b)]
Let $J\subset \tree(R)$ be a family of pairwise disjoint cubes. Then
\begin{equation}\label{eqkey7}
\mu\biggl(\,\bigcup_{Q\in J} K_Q^{1/2} \cap (R\setminus M\,B_Q)\biggr) \leq C(A,\tau,M)\,\ve\,
\mu(R).
\end{equation}
\end{enumerate}
\end{lemma}

\begin{proof}[Proof of (a)]
Let $P \in \DD_\mu(R)$ be such that $Q' \subsetneq P \subset R$, $P\subset K_Q^{3/4}$ and $\ell(P)\approx\dist(P,Q)$. Let $S\in\DD_\mu(R)$ be such that $Q\subsetneq S\subset R$, $\ell(S)\approx \frac1M\,\ell(P)$, and
$\dist(P,S)\approx \ell(P)$. For $M$ big enough and an appropriate choice of the implicit constants, we have
$$2B_P\subset K_x\quad \mbox{ for all $x\in 2B_S$.}$$
Therefore,
\begin{equation}\label{keylemma-keyineq}
\mu(2B_P)\,\mu(2B_S)\,\frac1{\ell(P)^n} \lesssim \iint_{\begin{subarray}{l}x\in 2B_{S}\\x-y\in K\\ \eta\,\ell(S)\leq |x-y|\leq \eta^{-1}\ell(S)
\end{subarray}} \frac1{|x-y|^n}\,d\mu(x)\,d\mu(y) = \mu(S)\,\EE_\mu(S),
\end{equation}
assuming $\eta$ small enough (depending on $M$).

Since $Q \in \tree(R)$ and $Q \subsetneq S$, we have $S\not\in\sss(R)$ and thus 
$$\Theta_\mu(2B_P) \approx \frac{\mu(2B_P)}{\ell(P)^n}  \lesssim \frac{\mu(S)}{\mu(2B_S)} \EE_\mu(S) \leq \EE_\mu(S) \leq \ve \Theta_\mu(2B_R).$$
Thus if $\ve$ is small enough, then $P \in LD_0(R)$. Since $Q' \subsetneq P$, it follows that $Q' \not\in \tree(R)$.
\end{proof}

\begin{proof}[Proof of (b)]
We can assume that the cubes of the family $J$ cover $R$. Otherwise we replace $J$ by a suitable enlarged family $J'$.

For a fixed $Q\in J$, if $M$ is big enough, we can cover $K_Q^{1/2}\cap (R\setminus MB_Q)$ by a family of cubes
$P\in\DD_\mu(R)$ such that
$P\subset K_Q^{3/4}$, $P\cap (R\setminus MB_Q)\neq\varnothing$, and $\ell(P)\approx\dist(P,Q)$. We denote by $I_Q$ this family. We assign a cube $S_{P,Q}\in\DD_\mu(R)$ to each $P\in I_Q$ such that $Q\subsetneq S_{P,Q}\subset R$, $\ell(S_{P,Q})\approx \frac1M\,\ell(P)$, and
$\dist(P,S_{P,Q})\approx \ell(P)$. As in the proof of \eqref{keylemma-keyineq}, for $M$ big enough, $\eta$ small enough, and an appropriate choice of the implicit constants, we have
$$\mu(2B_P)\,\mu(2B_{S_{P,Q}})\,\frac1{\ell(P)^n} \lesssim \mu(S_{P,Q})\,\EE_\mu(S_{P,Q}).$$
Since $Q\subsetneq S_{P,Q}$, we have $S_{P,Q}\not\in\sss(R)$ and thus 
$$\mu(2B_{S_{P,Q}})\,\frac1{\ell(P)^n}\approx_M \mu(2B_{S_{P,Q}})\,\frac1{\ell(S_{P,Q})^n}\approx_{A,\tau,M} \Theta_\mu(2B_R),$$
so
$$\Theta_\mu(2B_R)\,\mu(P) \lesssim_{A,\tau,M} \mu(S_{P,Q})\,\EE_\mu(S_{P,Q}).$$

Consider now a maximal subfamily $\AZ\subset\bigcup_{Q\in J} I_Q$, so that the cubes from $\AZ$ are pairwise disjoint and
$$\bigcup_{Q\in J} K_Q^{1/2} \cap (R\setminus M\,B_Q)\subset \bigcup_{P\in \AZ}P.$$
For each $P\in\AZ$ we choose a cube $S(P):=S_{P,Q}$, where $Q$ is such that $P\in I_Q$. The precise choice of $Q$ does not matter as long as $P\in I_Q$.
Observe that for each cube $S\in\DD_\mu$ there is at most a bounded number (depending on $M$) of cubes $P\in\AZ$ such that $S=S(P)$, taking into account that $\ell(S(P))\approx \frac1M\,\ell(P)$ and $\dist(P,S(P))\approx \ell(P)$.
Further, all the cubes $\{S(P)\}_{P\in\AZ}$ belong to $\tree(R)\setminus\sss(R)$.
As a consequence,
\begin{align}\label{eqj33}
\mu\biggl(\,\bigcup_{Q\in J} K_Q^{1/2} \cap (R\setminus M\,B_Q)\biggr) & \leq \sum_{P\in\AZ}\mu(P)\\
& \lesssim_{A,\tau,M}\frac1{\Theta_\mu(2B_R)} \sum_{P\in\AZ}\mu(S(P))\,\EE_\mu(S(P))\nonumber\\
& \lesssim_{A,\tau,M}\frac1{\Theta_\mu(2B_R)} \sum_{{\substack{S\in\tree(R)\setminus\sss(R)\\
\text{$S\supset Q$ for some $Q\in J$}}}}
\mu(S)\,\EE_\mu(S).\nonumber
\end{align}
By Fubini, we get
\begin{align*}
\sum_{{\substack{S\in\tree(R)\setminus\sss(R)\\
\text{$S\supset Q$ for some $Q\in J$}}}}\mu(S)\,\EE_\mu(S) & = \sum_{{\substack{S\in\tree(R)\setminus\sss(R)\\
\text{$S\supset Q$ for some $Q\in J$}}}}\,\sum_{Q\in J:Q\subset S}\mu(Q)\,\EE_\mu(S)\\
 & = \sum_{Q\in J}\mu(Q)\,\sum_{\substack{S\in\tree(R)\setminus\sss(R)\\
 Q\subset S\subset R}}\EE_\mu(S).
 \end{align*}
Note now that for any $Q\in J$,
$$\sum_{\substack{S\in\tree(R)\setminus\sss(R)\\
 Q\subset S\subset R}}\EE_\mu(S)\leq \ve\,\Theta_\mu(2B_R),$$
 because of the stopping condition involving the cubes from $\BCE(R)$.
Therefore,
$$\sum_{{\substack{S\in\tree(R)\setminus\sss(R)\\
\text{$S\supset Q$ for some $Q\in J$}}}}
\mu(S)\,\EE_\mu(S)\leq \ve\,\Theta_\mu(2B_R)\sum_{Q\in J}\mu(Q)\leq \ve\, \Theta_\mu(2B_R)\,\mu(R).$$
 From \rf{eqj33}
and the preceding estimate, the lemma follows.
\end{proof}

\vv

Denote
\[
G_R = R\setminus \bigcup_{Q\in\sss(R)} Q
\qquad\text{and}\qquad
\wt G_R = \bigcap_{k=1}^\infty \bigcup_{\substack{Q \in \tree(R)\\ \ell(Q) \leq A_0^{-k}}} 2M  B_Q
.
\]
and observe that $G_R \subset \wt G_R$. As an immediate consequence of the preceding lemma, we can show that $\wt G_R$ is contained in a Lipschitz graph.

\begin{lemma}\label{lemma:lip-gR-gR}\label{lemgr}
For all $x, y \in \wt G_R$, we have $x - y \not\in K^{1/2}$. Hence, $\wt G_R$ is contained in an $n$-dimensional Lipschitz graph with Lipschitz constant depending only on $s$.
\end{lemma}

\begin{proof}
 Suppose for contradiction that $x, y \in \wt G_R$ and $x-y \in K^{1/2}$. Let $Q,Q' \in  \tree(R)$ be such that $x \in 2M  B_Q$ and $y \in 2M  B_{Q'}$, with $\ell(Q), \ell(Q')$ so small that $Q' \cap (K_Q^{1/2} \setminus MB_Q) \neq \varnothing$ and $\ell(Q') \leq \frac{1}{M} \dist(Q', Q)$. By Lemma~\ref{keylemma}(a), it follows that $Q' \not\in \tree(R)$, which is a contradiction.
\end{proof}

\vv

\subsection{An algorithm to construct a Lipschitz graph close to the stopping cubes}\label{sec4.4}

Given $t>1$, we say that two cubes $Q,Q'\subset\DD_\mu$ are $t$-neighbors if
\begin{equation}\label{eqnei1}
t^{-1}\,\ell(Q')\leq\ell(Q)\leq t\,\ell(Q')
\end{equation}
and
\begin{equation}\label{eqnei2}
\dist(Q,Q')\leq t(\ell(Q) + \ell(Q')).
\end{equation}
We say that a family of cubes is $t$-separated if there is not any pair of cubes in this family which are $t$-neighbors.

Given a big constant $t>M$ to be fixed below, we denote by $\sep(R)$ a maximal $t$-separated subfamily of 
$\sss(R)$. It is easy to check that such subfamily exists.
Next, we define
\[
\wt\sep(R)
=
\{
Q \in \sep(R) : 2MB_Q \cap \wt G_R = \varnothing \text{ and } \nexists Q' \in \sep(R) \text{ such that } 2MB_{Q'} \subset 2MB_Q
\}.
\]

\begin{lemma}\label{lemma:sep-Q-not-in-Q}
Suppose $t$ is sufficiently large (depending on $M$). Then for all $Q, Q' \in \wt\sep(R)$, we have $Q' \not\subset MB_Q$.
\end{lemma}

\begin{proof}
Suppose $Q' \in \wt\sep(R)$ and $Q' \subset MB_Q$. We will show $Q \not\in \wt\sep(R)$. If $\ell(Q') > t^{-1} \ell(Q)$, then $Q' \subset MB_Q$ implies that $Q, Q'$ are $t$-neighbors if $t$ is sufficiently large. On the other hand, if $\ell(Q') \leq t^{-1} \ell(Q)$, then $Q' \subset MB_Q$ implies that $2MB_{Q'} \subset 2M B_Q$ if $t$ is sufficiently large. Hence, in both cases, we have $Q \not\in \wt\sep(R)$.
\end{proof}

\begin{lemma}\label{lemma:stop-sep-wtsep}\label{lemfac1}
The following holds:
\begin{itemize}
\item[(a)]
For each $Q\in\sss(R)$ there exists some cube $Q' \in\sep(R)$ which is $t$-neighbor of $Q$.
\item[(b)]
For each $Q\in\sep(R)$, at least one of the following is true:
\begin{itemize}
\item $2M B_Q \cap \wt G_R \neq \varnothing$.
\item There exists some $P\in \wt\sep(R)$ such that $P\subset 2M B_Q$.
\end{itemize}
\end{itemize}
\end{lemma}

\begin{proof}
The first statement is obvious from the maximality of the separated family $\sep(R)$.
For the second one, note first that the statement is clearly true if $Q\in\wt\sep(R)$.
If $Q\in\sep(R)\setminus\wt\sep(R)$ and $2MB_Q \cap \wt G_R = \varnothing$, then there exists another cube $Q_1\in\sep(R)$ such that %
$\ell(Q_1)\leq t^{-1}\ell(Q)$ and $2M B_{Q_1}\subset  2M B_{Q}$.

If $Q_1\in\wt\sep(R)$, then we take $P=Q_1$. Otherwise, since $2M B_{Q_1} \cap \wt G_R \subset 2MB_{Q} \cap  \wt G_R  = \varnothing$, there exists some cube $Q_2\in\sep(R)$ such that %
$\ell(Q_2)\leq t^{-1}\ell(Q_1)$ and $2M B_{Q_2}\subset  2M B_{Q_1}$.
Iterating this process, we will get a sequence of cubes
$Q\equiv Q_0,Q_1,\ldots Q_m$ such that $\ell(Q_j)\leq t^{-1}\ell(Q_{j-1})$ and
$ 2M B_{Q_j}\subset  2M B_{Q_{j-1}}$,
for $j=1,\ldots,m$.

If the process does not terminate, then $\bigcap_{j=0}^\infty  2M B_{Q_j}$ is nonempty. By definition of $\wt G_R$, we have $\bigcap_{j=0}^\infty  2M B_{Q_j} \subset \wt G_R$, which contradicts our assumption that $2MB_Q \cap \wt G_R =\varnothing$. Hence, this process terminates at some $Q_m$ (i.e., $Q_m\in\wt\sep(R)$). We take $P=Q_m$, and obtain $P\subset  2MB_P \subset 2M B_Q$.
\end{proof}

\begin{lemma}\label{lemma:lip-sep-sep}
Assume that $t$ is chosen big enough (depending on $M$, but not on $A$, $\tau$, or $\ve$). Then:
\begin{enumerate}
\item[(a)] For all $Q, Q' \in \wt\sep(R)$, we have 
 \begin{equation}\label{eq:Q-cap-KQ-empty}
Q \cap K_{Q'}^{1/2} = Q' \cap K_{Q}^{1/2} = \varnothing.
 \end{equation}
\item[(b)] For all $x \in \wt G_R$ and for all $Q \in \wt\sep(R)$, we have 
 \begin{equation}
x \not\in K_Q^{1/2} \text{ and } Q \cap K_x^{1/2} = \varnothing.
 \end{equation}
\end{enumerate}
\end{lemma}

\begin{proof}[Proof of (a)]
Suppose that \eqref{eq:Q-cap-KQ-empty} fails. Observe that this implies that both $Q \cap K_{Q'}^{1/2}$ and $Q' \cap K_{Q}^{1/2}$ are nonempty. Suppose that $\ell(Q)\leq \ell(Q')$.  
 Since $Q$ and $Q'$ are not $t$-neighbors, then we must have $\ell(Q')\leq t^{-1}\,\ell(Q)$.

We claim that $Q' \subset MB_{Q}$. Suppose not. Then $Q' \cap (K_{Q}^{1/2} \setminus MB_{Q}) \neq \varnothing$, and $\ell(Q')\leq t^{-1}\,\ell(Q) \leq \frac{1}{M} \dist(Q, Q')$. Hence it follows that $Q' \not\in \tree(R)$, a contradiction. This shows that $Q' \subset MB_{Q}$. But that contradicts Lemma~\ref{lemma:sep-Q-not-in-Q}.
\end{proof}

\begin{proof}[Proof of (b)]
Suppose for contradiction that $x \in K_Q^{1/2}$. We claim that $x \in MB_Q$. Suppose not, so that $x \in K_{Q}^{1/2}  \setminus MB_{Q}$. Let $Q'\in \tree(R)$ be such that $x \in 2M B_{Q'}$ with $\ell(Q')$ so small that $\ell(Q') \leq \frac{1}{M} \dist(Q', Q)$. By this inequality and the fact that $Q' \cap (K_Q^{1/2} \setminus MB_Q) \neq \varnothing$, it follows from Lemma~\ref{keylemma}(a) that $Q' \not\in \tree(R)$, which is a contradiction. Hence $x \in MB_Q$, so $2MB_Q \cap \wt G_R \neq \varnothing$. But this contradicts $Q \in \wt\sep(R)$.
\end{proof}

\begin{lemma}\label{lemc0q}
Let $\Lambda_0>0$ be big enough, depending on $M$ and $t$.
There is a Lipschitz graph $\Gamma_R$ with slope depending only on $s$ such that $\widetilde G_R \subset \Gamma_R$ and $\Lambda_0 B_Q \cap \Gamma_R \neq \varnothing$ for every $Q \in \tree(R)$.
\end{lemma}

\begin{proof}
For each $Q \in \wt\sep(R)$, pick a point $z_Q \in Q$. Let $F = \{z_Q : Q \in \wt\sep(R)\} \cup \wt G_R$. By Lemma \ref{lemma:lip-gR-gR} and Lemma \ref{lemma:lip-sep-sep}, it follows that
\[
x-y \not\in K^{1/2}
\text{ for all } x,y \in F.
\]
Hence there is a Lipschitz graph $\Gamma_R$ containing
$
F.
$

By Lemma \ref{lemma:stop-sep-wtsep}, any cube $P\in \sss(R)$ is $t$-neighbor of some cube $P'\in\sep(R)$ and $2M B_{P'}\cap F\neq\varnothing$. So if $Q\in\tree(R)$, then either $Q$ intersects $\wt G_R$ or $Q$ contains some cube $P\in\sss(R)$ and thus
$$\dist(Q,\Gamma_R)\leq \dist(P,\Gamma_R)\leq C(t)\ell(P)\leq C(t)\,\ell(Q).$$
which implies $\Lambda_0 B_Q \cap \Gamma_R \neq \varnothing$.
\end{proof}

\vv

\subsection{The small measure of the low density set}

The goal in this section is to estimate the total measure of the cubes in the low density set.

\begin{lemma}\label{lemld*}
We have
$$\sum_{Q\in\LD(R)}\mu(Q) \leq c\,(\tau+C(A,\tau,M)\ve)\,\mu(R).$$
\end{lemma}

To prove Lemma~\ref{lemld*},  we will construct an auxiliary $n$-dimensional Lipschitz graph, by arguments quite similar to the ones for
$\Gamma_R$. 
We denote by $\wt \sss(R)$ the subfamily of cubes $Q\in\sss(R)$ such that
$$Q\not\subset \bigcup_{P\in \sss(R)} K_P^{1/2} \cap (R\setminus M\,B_P),$$
so that, by Lemma \ref{keylemma}(b),
\begin{equation}\label{eqstop3}
\mu\biggl(\,\bigcup_{Q\in \sss(R)\setminus \wt\sss(R)} Q\biggr) \leq C(A,\tau,M)\,\ve\,
\mu(R).
\end{equation}

We claim that we can choose a subfamily $\LD_{\sep}(R)\subset \LD(R)\cap\wt\sss(R)$ which is $t$-separated and such that
\begin{equation}\label{eqsum941}
\sum_{Q\in\LD(R)\cap\wt\sss(R)}\mu(Q) \leq C(t)\sum_{Q\in\LD_{\sep}(R)}\mu(Q).
\end{equation}
To this end we argue as follows: let $\LD_1(R)$ be a maximal $t$-separated subfamily of $\LD(R)\cap\wt\sss(R)$. Let 
$\LD_2(R)$ be a maximal $t$-separated subfamily of $\LD(R)\cap\wt\sss(R)\setminus \LD_1(R)$. By induction, let
$\LD_j(R)$ be a maximal $t$-separated subfamily of $\LD(R)\cap\wt\sss(R)\setminus (\LD_1(R)\cup\ldots\cup\LD_{j-1}(R))$.
It turns out that there is bounded number $N_0$ of non-empty families $\LD_j(R)$, with $N_0$ depending on $t$. Indeed, if $Q\in\LD_j(R)$, then $Q$ is a $t$-neighbor of some cubes 
$Q_1\in\LD_1(R)$, $Q_2\in\LD_2(R)$,\ldots, $Q_{j-1}\in\LD_{j-1}(R)$, by the maximality of $\LD_k(R)$ for $k=1,\ldots,j-1$. Since the number
of $t$-neighbors of any cube has some bound depending on $t$, we get $j\leq N_0$.
Now we just let $\LD_{\sep}(R)$ be the family $\LD_j(R)$ for which $\sum_{Q\in\LD_j(R)}\mu(Q)$ is maximal, and then we have
$$
\sum_{Q\in\LD(R)\cap\wt\sss(R)}\mu(Q) \leq N_0\sum_{Q\in\LD_{\sep}(R)}\mu(Q),
$$
which proves our claim.

 Next, we modify the family $\LD_{\sep}(R)$ as follows:
if there are two cubes $Q,Q'\in\LD_\sep(R)$ such that
\begin{equation}\label{eqqq00} 
1.1B_Q\cap 1.1B_{Q'}\neq\varnothing\quad \mbox{and}\quad\ell(Q)<\ell(Q'),
\end{equation}
then we eliminate $Q'$. 
We denote by $\wt\LD(R)$ the resulting family after eliminating all cubes $Q'$ of this type in $\LD_\sep(R)$. We have the following variant of Lemma \ref{lemfac1}(b).

\begin{lemma}
For each $Q\in\LD_\sep(R)$, at least one of the following is true:
\begin{itemize}
\item $1.2 B_Q \cap \wt G_R \neq \varnothing$ 
\item There exists some $P\in \wt\LD(R)$ such that $P\subset 1.2 B_Q$.
\end{itemize}
\end{lemma}

\begin{proof}
If $Q,Q'\in\LD_\sep(R)$ satisfy \rf{eqqq00}, then \rf{eqnei2} holds, and thus \rf{eqnei1} must fail.
Therefore, $Q$ must be much smaller than $Q'$, and so
\[%
\ell(Q)\leq t^{-1}\ell(Q')\quad\mbox{ and }\quad 1.2B_Q\subset 1.2 B_{Q'}\quad \mbox{if $\ell(Q)<\ell(Q')$},
\]%
assuming $t$ big enough to guarantee the last inclusion. Now we can copy the proof of Lemma~\ref{lemfac1}(b) with $2M$ replaced everywhere by $1.2$.
\end{proof}

For each $Q\in\wt\LD(R)$ we choose a point
$$w_Q\in Q \setminus \bigcup_{P\in \sss(R)} K_P^{1/2} \cap (R\setminus M\,B_P).$$

\begin{lemma}
There exists an $n$-dimensional Lipschitz graph $\Gamma_0$ which passes through every point $w_P$, $P\in\wt\LD(R)$.
\end{lemma}

\begin{proof}
It is enough to show that for any pair $w_Q$, $w_{Q'}$, with $Q,Q'\in\wt\LD(R)$, $Q\neq Q'$,
 we have 
 \begin{equation}\label{eqcone44}
 w_Q-w_{Q'}\not\in K^{1/2}.
 \end{equation}
 To show this, suppose that $\ell(Q)\leq \ell(Q')$. By the construction of the points $w_P$, $P\in\wt\LD(R)$, it follows that
 $$w_{Q'}\not\in  K_Q^{1/2} \cap (R\setminus M\,B_Q)\quad\mbox { and }\quad w_{Q}\not\in  K_{Q'}^{1/2} \cap (R\setminus M\,B_{Q'}),$$
which implies that 
$$w_{Q'}\not\in K_{w_Q}^{1/2}\setminus B(w_Q,c_1M\ell(Q)) \quad\mbox{ and }\quad w_{Q}\not\in K_{w_{Q'}}^{1/2}\setminus B(w_{Q'},c_1M\ell(Q')),$$
for some $c_1\approx1$. 

So to conclude the proof of \rf{eqcone44} it suffices to show that 
\begin{equation}\label{eqcone45}
w_{Q'}\not\in B(w_Q,c_1M\ell(Q)).
\end{equation}
To this end, notice that if $w_{Q'}\in B(w_Q,c_1M\ell(Q))$, then 
$$\dist(Q,Q')\leq|w_Q-w_{Q'}|\leq c_1M\ell(Q)\leq t(\ell(Q)+ \ell(Q')),$$
assuming $t=t(M)$ big enough. Since $Q$ and $Q'$ are not $t$-neighbors, we must have
$$\ell(Q)\leq t^{-1}\,\ell(Q').$$
Together with the fact that $1.1B_Q\cap 1.1B_{Q'}=\varnothing$, and recalling that 
$w_Q\in Q\subset B_Q$ and $w_{Q'}\in Q'\subset B_{Q'}$, this implies that
$$|w_Q-w_{Q'}|\geq \frac1{10}\,r(B_{Q'}) \geq \frac {c\,t}{10}\,\,r(B_{Q}) > c_1\,M\,\ell(Q)$$
if $t(M)$ is big enough again.
So \rf{eqcone45} holds, and the lemma follows.
\end{proof}

\vv
\begin{proof}[Proof of Lemma~\ref{lemld*}]
Consider the family of balls $\{ 1.5 B_Q\}_{Q\in\LD_\sep(R)}$. By the covering Theorem 9.31 from \cite{Tolsa-book},
there exists a subfamily $\FF\subset \LD_\sep(R)$ such that:
\begin{itemize}
\item[(i)] $\bigcup_{Q\in\LD_\sep(R)}  1.5 B_Q \subset\bigcup_{Q\in\FF} 2 B_Q$,
\item[(ii)] $\sum_{Q\in\FF} \chi_{ 1.5 B_Q}\leq C$.
\end{itemize}
Then
\begin{align*}
\sum_{Q\in\LD_\sep(R)}\mu(Q) \leq \sum_{Q\in\FF}\mu(2 B_Q)\leq c\,\tau\,\Theta_\mu(2B_R)\sum_{Q\in\FF}r(B_Q)^n.
\end{align*}
Recall now that for each $B_Q$, with $Q\in\FF\subset\LD_\sep(R)$, there exists some point $w_P\in\Gamma_0\cap 1.2 B_Q$ or some point $x \in \wt G_R \cap 1.2B_Q \subset \Gamma_R \cap 1.2B_Q$.
Then we have
$$\HH^n( 1.5 B_Q\cap(\Gamma_0\cup\Gamma_R))\approx r(B_Q)^n.$$
So using the property (ii) of the covering, we obtain
$$\sum_{Q\in\FF}r(B_Q)^n\lesssim \sum_{Q\in\FF}\HH^n( 1.5 B_Q\cap(\Gamma_0\cup\Gamma_R)) \leq C\,\HH^n(2B_R \cap (\Gamma_0\cup\Gamma_R))\leq C'\,\ell(R)^n.$$
Thus,
$$\sum_{Q\in\LD_\sep(R)}\mu(Q)\leq c\,\tau\,\Theta_\mu(2B_R)\,\ell(R)^n \leq c\,\tau\,\mu(R).$$
Together with \rf{eqsum941}, this yields
$$
\sum_{Q\in\LD(R)\cap\wt\sss(R)}\mu(Q) \leq C\,\tau\,\mu(R),
$$
with $C$ depending on $t$.
Finally to conclude the lemma, we just take into account that, by \eqref{eqstop3}, we have
$$
\sum_{Q\in\LD(R)\setminus\wt\sss(R)}\mu(Q) \leq C(A,\tau,M)\,\ve\,\mu(R).
$$
\end{proof}

\vv

\subsection{The approximate Lipschitz graph}

In the next lemma we gather some of the previous results and estimates.

\begin{lemma}\label{lem47}
Let $R\in\DD^{db}_\mu$, and suppose that $\tau,\eta,\ve$ are small enough and $\ve\ll\tau$. 
Then there exists an $n$-dimensional Lipschitz graph $\Gamma_R$ whose slope is
bounded above by some constant depending only on $s$
such that the following holds:
\begin{itemize}

\item[(a)] $R\setminus \bigcup_{Q\in\sss(R)}Q \subset \Gamma_R$.
\vv

\item[(b)] There exists some constant $\Lambda_0>1$ such that for all $Q\in\tree(R)$, $$\Lambda_0\,B_Q\cap
\Gamma_R\neq \varnothing.$$

\item[(c)] We have:
\begin{equation}\label{eqcc1}
\sum_{Q\in\LD(R) }\mu(Q) \leq \tau^{1/2}\,\mu(R),
\end{equation}
and also
\begin{equation}\label{eqcc2}
\sum_{Q\in\BCE(R)} \mu(Q) \leq \frac1{\ve\,\Theta_\mu(2B_R)}
\sum_{S\in\tree(R)} \EE_\mu(S)\,\mu(S).
\end{equation}
\end{itemize}
\end{lemma}

\begin{proof}
The statement (a) follows from Lemma \ref{lemgr}, and (b) from Lemma \ref{lemc0q}.
On the other hand, the estimate \rf{eqcc1} follows from the analogous one proved in Lemma \ref{lemld*} choosing $\ve$ and $\tau$ suitably small.
Finally, concerning \rf{eqcc2}, recall that if $Q\in\BCE(R)$, then
$$\sum_{S\in\DD_\mu:Q\subset S \subset R} \EE_\mu(S)
\geq \ve\,\Theta_\mu(2B_R).$$
Therefore,
\begin{align*}
\Theta_\mu(2B_R)\,\sum_{Q\in\BCE(R)} \mu(Q) & \leq \frac1\ve
 \sum_{Q\in\BCE(R)} \mu(Q) \sum_{S\in\DD_\mu:Q\subset S \subset R} \EE_\mu(S)\\
& = \frac1\ve
\sum_{S\in\tree(R)} \EE_\mu(S)
\sum_{Q\in\BCE(R):Q\subset S} \mu(Q) \\
& \leq \frac1\ve
\sum_{S\in\tree(R)} \EE_\mu(S)\,\mu(S).
\end{align*}
\end{proof}

\vv

\section{The family of $\ttt$ cubes}

\subsection{The family $\ttt$}
We are going to construct a family of cubes $\ttt\subset\DD_\mu^{db}$ inductively. To this end, we need to introduce some additional notation.
Given a cube $Q\in\DD_\mu$, we denote by $\MD(Q)$ the family  of maximal cubes 
(with respect to inclusion) from $\DD_\mu^{db}(Q)\setminus\{Q\}$. Recall that, by Lemma \ref{lemcobdob}, this family covers $\mu$-almost
all of $Q$. Moreover, by Lemma \ref{lemcad23} it follows that if $P\in\MD(Q)$, then $\Theta_\mu(2B_P)
\leq c\,\Theta_\mu(2B_Q)$. 
Given $R\in\DD^{db}_\mu$, we denote
$$\nex(R) = \bigcup_{Q\in\sss(R)} \MD(Q).$$
By the construction above, it is clear that the cubes in $\nex(R)$ are different from $R$ (because $\MD(Q)
\neq \{Q\}$).

For the record, notice that if $P\in\nex(R)$, then
\begin{equation}\label{eqdens***}
\Theta_\mu(2B_S)\leq c(A)\,\Theta_\mu(2B_R) \quad\mbox{for all $S\in\DD_\mu$ such that $P\subset S\subset R$.}
\end{equation}

We are now ready to construct the aforementioned family $\ttt$. We will have
$\ttt=\bigcup_{k\geq0}\ttt_k$. First we set
$$\ttt_0=\{R_0\}.$$
(Recall that $R_0\equiv\supp\mu$.) 
Assuming $\ttt_k$ has been defined, we set
$$\ttt_{k+1}  = \bigcup_{R\in\ttt_k} \nex(R).$$
Note that the families $\nex(R)$, with $R\in\ttt_k$, are pairwise disjoint. 
\vv

\subsection{The family of cubes $ID$}

We distinguish a special type of cubes from $\ttt$. For $R \in \ttt$, we write $R\in ID$ (increasing density) if
$$\mu\biggl(\,\bigcup_{Q\in\HD(R)} Q\biggr)\geq \frac12  \,\mu(R).
$$

\vv

\begin{lemma}\label{lemID}
Suppose that $A$ is big enough. If $R\in ID$, then
\begin{equation}\label{eqaki33}
\Theta_\mu(2B_R)\,\mu(R) \leq \frac12 \sum_{Q\in \nex(R)}\Theta_\mu(2B_Q)\,\mu(Q).
\end{equation}
\end{lemma}

\begin{proof}
 Recalling that $\Theta_\mu(2B_Q)\geq A\,\Theta_\mu(2B_R)$ %
 for every $Q\in{\HD}(R)$,
we deduce that
 $$\Theta_\mu(2B_R)\,\mu(R) \leq 2 \sum_{Q\in {\HD}(R)}\Theta_\mu(2B_R)\,\mu(Q)
\leq 2A^{-1}\sum_{Q\in {\HD}(R)}\Theta_\mu(2B_Q)\,\mu(Q).$$
Since the cubes from ${\HD}(R)$ belong to $\DD_\mu^{db}$, it follows immediately from Lemma \ref{lemcubdob*} that for any $Q\in\HD(R)$,
\begin{equation*}
\Theta_\mu(2B_Q)\,\mu(Q)\lesssim \sum_{P\in\nex(R):P\subset Q}\Theta_\mu(2B_P)\,\mu(P),
\end{equation*}
and then  it is clear that \rf{eqaki33} holds if $A$ is taken big enough.
\end{proof}
\vv

\section{The packing condition}

\begin{lemma}\label{lemkey62}
Suppose that 
\begin{equation}\label{eqcreixr0} 
 \mu(B(x,r))\leq c_0\,r^n\quad \mbox{ for all
$x\in\supp\mu$, $r>0$.}
\end{equation}
For all $S\in\ttt$ we have
\begin{equation}\label{eqsum441}
\sum_{R\in\ttt:R\subset S}\Theta_\mu(2B_R)\,\mu(R)\lesssim_{\ve,\eta,c_0} \mu( S)
+ \iint_{\begin{subarray}{l}x\in 2B_S\\x-y\in K\end{subarray}} \frac1{|x-y|^n}\,d\mu(x)\,d\mu(y)
,
\end{equation}
assuming that the constants $A,\tau,\ve$, and $\eta$ 
have been chosen suitably.
\end{lemma}

\begin{proof}
We denote $\ttt( S)= \ttt\cap\DD_\mu( S)$ and $\ttt_j( S)= \ttt_j\cap\DD_\mu( S)$.
For a given $k\geq 0$, we also write 
$$\ttt_0^k( S) = \bigcup_{0\leq j\leq k}\ttt_j( S) ,$$
and also
$$ID_0^k = ID \cap \ttt_0^k( S).$$

To prove \rf{eqsum441}, first we deal with the cubes from the $ID$ family.
By Lemma \ref{lemID}, for every $R\in ID$ we have
$$\Theta_\mu(2B_R)\,\mu(R) \leq \frac12 \sum_{Q\in\nex(R)} \Theta_\mu(2B_Q)\,\mu(Q)$$
 and hence we obtain
$$
\sum_{R\in ID_0^k} \Theta_\mu(2B_R)\,\mu(R) \leq \frac12 \sum_{R\in ID_0^k}\sum_{Q\in \nex(R)}\Theta_\mu(2B_Q)\,\mu(Q)
\leq \frac12\sum_{Q\in \ttt_0^{k+1}( S) }\Theta_\mu(2B_Q)\,\mu(Q),
$$
because the cubes from $\nex(R)$ with $R\in\ttt_0^k( S) $ belong to $\ttt_0^{k+1}( S) $. 
Thus,
\begin{align*}
\sum_{R\in \ttt_0^k( S) } \!\!\Theta_\mu(2B_R)\,\mu(R) & = \sum_{R\in \ttt_0^k( S) \setminus ID_0^k} \Theta_\mu(2B_R)\,\mu(R)+
\sum_{R\in ID_0^k} \Theta_\mu(2B_R)\,\mu(R)\\
& \leq \!\!\sum_{R\in \ttt_0^{k}( S) \setminus ID_0^k}\!\!\!\Theta_\mu(2B_R)\,\mu(R) +
\frac12\!\sum_{R\in \ttt_0^{k}( S)}\!\!\Theta_\mu(2B_R)\,\mu(R) +
 C\,c_0\,\mu( S),
\end{align*}
where, for the last inequality, we took into account that $\Theta_\mu(2B_R) \leq C\,c_0$ for every $R\in\ttt_{k+1}( S) $ because of the assumption \rf{eqcreixr0}.
Using that $$\sum_{R\in \ttt_0^{k}( S)}\Theta_\mu(2B_R)\,\mu(R)\leq (k+1)\,C\,c_0\,\mu(S)<\infty,$$ we deduce 
$$\sum_{R\in \ttt_0^k( S) } \Theta_\mu(2B_R)\,\mu(R)  \leq 2\sum_{R\in \ttt_0^{k}( S) \setminus ID_0^k}\Theta_\mu(2B_R)\,\mu(R) +
 C\,c_0\,\mu(S).$$
Letting $k\to\infty$, we derive
\begin{equation}\label{eqaaii}
\sum_{R\in \ttt( S)} \Theta_\mu(2B_R)\,\mu(R)\leq 2\sum_{R\in \ttt( S) \setminus ID} \Theta_\mu(2B_R)\,\mu(R)+  C \,c_0\,\mu( S).
\end{equation}

To estimate the first term on the right hand side of \rf{eqaaii} we use the fact that, for $R\in\ttt( S) \setminus ID$, 
we have
$$
\mu\biggl(R\setminus \bigcup_{Q\in \HD(R)} Q\biggr)\geq \frac12\,\mu(R),$$
and then using Lemma \ref{lemcobdob}, we get 
\begin{align}\label{eqdd9}
\mu(R)& \leq 
2\,\mu\biggl(R\setminus \bigcup_{Q\in \sss(R)} Q\biggr) 
+ 2
\,\mu\biggl(
\bigcup_{Q\in\sss(R)\setminus \HD(R) }Q\biggr)\\
& = 2 \,\mu\biggl(R\setminus \bigcup_{Q\in \nex(R)} Q\biggr) 
+2\sum_{Q\in\LD(R)}\mu(Q) + 2\sum_{Q\in\BCE(R)}\mu(Q)
.\nonumber
\end{align}
Recall now that, by \rf{eqcc1},
$$\sum_{Q\in\LD(R)}\mu(Q)\leq \tau^{1/2}\mu(R).$$
Choosing  $\tau\leq1/16$, say, from \rf{eqdd9} we infer that
$$\mu(R) \leq 4 \,\mu\biggl(R\setminus \bigcup_{Q\in \nex(R)} Q\biggr) 
+ 4\sum_{Q\in\BCE(R)}\mu(Q).
$$
So we deduce that
\begin{align}\label{eqal163}
\sum_{R\in \ttt( S) \setminus ID} \Theta_\mu(2B_R)\,\mu(R)
 & \leq 4 
\sum_{R\in\ttt( S)} \Theta_\mu(2B_R)\,\mu\biggl(R\setminus \bigcup_{Q\in \nex(R)} Q\biggr)
\nonumber\\
&\quad
+4\sum_{R\in\ttt( S)}\Theta_\mu(2B_R)\sum_{Q\in\BCE(R)}\mu(Q).
\end{align}
To deal with the first sum on the right hand side above, we take into account that the sets $R\setminus \bigcup_{Q\in \nex(R)} Q$, with $R\in\ttt( S)$, are pairwise disjoint, and also that $\Theta_\mu(2B_R)
\leq C\,c_0$, by the 
condition \rf{eqcreixr0}.
 Then we get
 \begin{equation}\label{eqfacgg1}
\sum_{R\in\ttt(S)} \Theta_\mu(2B_R)\,\mu\biggl(R\setminus \bigcup_{Q\in \nex(R)} Q\biggr)\leq C\, c_0\,\mu(S).
\end{equation}

To deal with the second sum in \rf{eqal163}, we use \rf{eqcc2} to obtain
$$\sum_{R\in\ttt(S)}\Theta_\mu(2B_R)\sum_{Q\in\BCE(R)}\mu(Q) \leq \frac1\ve\sum_{R\in\ttt(S)}
\sum_{P\in\tree(R)} \EE_\mu(P)\,\mu(P)\leq \frac1\ve\sum_{P\in\DD_\mu(S)}
 \EE_\mu(P)\,\mu(P).
$$
Denote by $\ell_k$ the side length of the cubes from $\DD_{\mu,k}$.
By the definition of $\EE_\mu(P)$ (see \eqref{eq:defn-EEmuQ}) and the finite overlapping of the balls $2B_P$ among the cubes $P$ of the same
generation, we get
\begin{align*}
\sum_{P\in\DD_\mu(S)}
 \EE_\mu(P)\,\mu(P) &= \sum_{k}\sum_{P\in\DD_{\mu,k}(S)}
 \iint_{\begin{subarray}{l}x\in 2B_P\\x-y\in K\\ \eta\,\ell(P)\leq |x-y|\leq \eta^{-1}\ell(P)
\end{subarray}} \frac1{|x-y|^n}\,d\mu(x)\,d\mu(y)\\
& \lesssim 
\sum_{k}
\iint_{\begin{subarray}{l}x\in 2B_S\\x-y\in K\\ \eta\,\ell_k\leq |x-y|\leq \eta^{-1}\ell_k
\end{subarray}} \frac1{|x-y|^n}\,d\mu(x)\,d\mu(y)\\
& \lesssim_\eta
\iint_{\begin{subarray}{l}x\in 2B_S\\x-y\in K\end{subarray}} \frac1{|x-y|^n}\,d\mu(x)\,d\mu(y).
\end{align*}
Therefore,
$$\sum_{R\in\ttt(S)} \Theta_\mu(2B_R)\,\sum_{Q\in\BCE(R)} \mu(Q) \lesssim_{\ve,\eta}
\iint_{\begin{subarray}{l}x\in 2B_S\\x-y\in K\end{subarray}} \frac1{|x-y|^n}\,d\mu(x)\,d\mu(y).$$
Together with \rf{eqaaii}, \rf{eqal163}, and \rf{eqfacgg1},
this yields \rf{eqsum441}.
\end{proof}
\vv

\section{Proof of Lemma \ref{lemcorona}}\label{sec8}

We have to show that the family $\ttt$ satisfies the properties stated  in Lemma  \ref{lemcorona}.
By the definition of the family $\nex(R)$ and Lemma \ref{lemcobdob}, we have
$$\bigcup_{Q\in\sss(R)}Q = \bigcup_{Q\in\nex(R)} Q \qquad\text{up to a set of $\mu$-measure $0$.}$$
Thus, by Lemma \ref{lem47}(a),
$\mu$-almost all of 
$R\setminus \bigcup_{Q\in\nex(R)} Q$ is contained in $\Gamma_R$, and we have verified property (a) of Lemma \ref{lemcorona}.

Next we deal with the property (b). Given $Q\in\nex(R)$ we have to check that there exists some $\wt Q\in\DD_\mu$ such that $\delta_\mu(Q,\wt Q)\leq c\,\Theta_\mu(2B_R)$ and $2B_{\wt Q}\cap\Gamma_R\neq\varnothing$. Let $Q'\in\sss(R)$ such that
$Q\subset Q'$. By Lemma \ref{lem47}(b), there exists some constant $\Lambda_0>1$ such that $\Lambda_0\,B_{Q'}\cap
\Gamma_R\neq \varnothing.$
This implies that there exists one cube $\wt Q\in\tree(R)$ such that $Q' \subset \wt Q$, $\ell(\wt Q)\approx\ell(Q')$, and 
$2B_{\wt Q}\cap
\Gamma_R\neq \varnothing$.
We split
$$\delta_\mu(Q,\wt Q) = \int_{2B_{\wt Q}\setminus 2B_{Q'}}\frac1{|y-x_Q|^n}\,d\mu(y) + 
\int_{2B_{Q'}\setminus 2B_{Q}}\frac1{|y-x_Q|^n}\,d\mu(y).$$
To estimate the first integral we use the fact that $|y-x_Q|\approx\ell(Q')\approx\ell(\wt Q)$ in the domain of integration, and we derive
$$\int_{2B_{\wt Q}\setminus 2B_{Q'}}\frac1{|y-x_Q|^n}\,d\mu(y)\lesssim \Theta_\mu(2B_{\wt Q})\lesssim_A \Theta_\mu(2B_R).$$
To estimate the second one we take into account that there are no doubling cubes strictly between $Q$ and $Q'$.
Then from Lemma \ref{lemcad23} and standard estimates, it easily follows that
$$\int_{2B_{Q'}\setminus 2B_{Q}}\frac1{|y-x_Q|^n}\,d\mu(y)\lesssim \Theta_\mu(100B(Q')).$$
If $Q'$ is doubling, then $\Theta_\mu(100B(Q'))\lesssim \Theta_\mu(2B_{Q'})\lesssim A\,\Theta_\mu(2B_R)$.
Otherwise $Q'\neq R$ and the parent of $Q'$, which we denote by $\wh{Q'}$, belongs to $\tree(R)\setminus\sss(R)$. Thus, by Lemma \ref{lemdobbb1},
$$\Theta_\mu(100B(Q'))\lesssim \Theta_\mu(2B_{\wh{Q'}})\lesssim A\,\Theta_\mu(2B_R),$$
taking into account that $100B(Q'))\subset2B_{\wh{Q'}}$ for the first inequality (since $A_0\gg1$ in the David-Mattila lattice).
Hence in any case, $\delta_\mu(Q,\wt Q)\lesssim_A\Theta_\mu(2B_R)$ and (b) in Lemma \ref{lemcorona} holds.

Next, we observe that (c) in Lemma \ref{lemcorona} follows from Lemma 
\ref{lemdobbb1} in case that $Q\in\tree(R)$, and from \rf{eqdens***} otherwise. 

Finally, the packing condition \rf{eqpack*1} has been proved in Lemma \ref{lemkey62}.

\vv

\section{Application to curvature, Riesz transforms, and capacities}

\subsection{Curvature of measures and Riesz transforms}

To estimate the curvature of $\mu$ we will use the following result:

\begin{lemma}\label{lemcorona2}
Let $\mu$ be a measure satisfying the growth condition \rf{eqgrowth00}.
Suppose that there exists a family $\ttt\subset\DD_\mu^{db}$ as in Section \ref{section:corona} %
such that, for all $R\in\ttt$, there exists an $n$-dimensional Lipschitz graph $\Gamma_R$ whose slope is
uniformly bounded above by some constant independent of $R$ such that:
\begin{itemize}
\item[(a)] $\mu$-almost all $\good(R)$ is contained in $\Gamma_R$.

\item[(b)] For all $Q\in\nex(R)$ there exists another cube $\wt Q\in\DD_\mu$ such that $\delta_\mu(Q,\wt Q)\leq c\,\Theta_\mu(2B_R)$ and $2B_{\wt Q}\cap\Gamma_R\neq\varnothing$.

\item[(c)] For all $Q\in\tr(R)$, $\Theta_\mu(2B_Q)\leq c\,\Theta_\mu(2B_R)$.
\end{itemize}
In the case $n=1$, we have:
\begin{equation}\label{eqcurv*1} 
c^2(\mu)\lesssim 
\sum_{R\in\ttt}\Theta_\mu(2B_R)^2\,\mu(R),
\end{equation}
and for any integer $n\in (0,d)$,
\begin{equation}\label{eqriesz*1} 
\sup_{\ve>0}\|\RR^n_\ve\mu\|_{L^2(\mu)}^2\lesssim
\sum_{R\in\ttt}\Theta_\mu(2B_R)^2\,\mu(R),
\end{equation}
with the implicit constant depending only on $c_0$ in both estimates.
\end{lemma}

For $n=1,d=2$, a version of this result which uses the usual dyadic lattice of $\R^2$ instead of the David-Mattila lattice is proven in \cite{Tolsa-bilip}. For arbitrary $n,d$,
another version in terms of the David-Mattila lattice is shown in \cite{Gi}, which in fact is valid for other singular
integral operators with odd kernel, besides the Cauchy and Riesz transforms. 

\vv

By combining Lemmas \ref{lemcorona} and \ref{lemcorona2} we obtain the following result.

\begin{theorem}\label{teocurv}
Given $V_0\in G(d,d-n)$ and $s>0$, consider the cone $K:=K(V_0,s)\subset\R^d$.
Let $\mu$ be a Borel measure in $\R^d$ satisfying the growth condition
$$\mu(B(x,r))\leq c_0\,r^n\quad\mbox{ for all $x\in\R^d$, $r>0$.}$$
In the case $n=1$
we have
$$c^2(\mu)\lesssim \|\mu\| + \iint_{x-y\in K}\frac1{|x-y|}\,d\mu(x)\,d\mu(y),$$
and for any integer $n\in(0,d)$,
$$\sup_{\ve>0}\|\RR^n_\ve\mu\|_{L^2(\mu)}^2\lesssim \|\mu\| + \iint_{x-y\in K}\frac1{|x-y|^n}\,d\mu(x)\,d\mu(y).$$
The implicit constants in both inequalities depend only on $d$, $n$, $c_0$, and $s$.
\end{theorem}

\begin{proof}
From Lemmas \ref{lemcorona} and \ref{lemcorona2}, in the case $n=1$ we deduce the following:
\begin{align*}
c^2(\mu)&\lesssim \|\mu\| + \sum_{R\in\ttt}\Theta_\mu(2B_R)^2\,\mu(R)\lesssim \|\mu\| +\sum_{R\in\ttt}\Theta_\mu(2B_R)\,\mu(R)\\
&\lesssim \|\mu\| + \iint_{x-y\in K}\frac1{|x-y|}\,d\mu(x)\,d\mu(y),
\end{align*}
and analogously, for any integer $n\in(0,d)$, with $\sup_{\ve>0}\|\RR^n_\ve\mu\|_{L^2(\mu)}^2$ instead of $c^2(\mu)$.
\end{proof}

\vv

\subsection{Analytic capacity}

To prove Theorem \ref{teo1}, recall that by Theorem \ref{teogam}, for any compact set $E\subset\C$ we have
\begin{equation}\label{eqgam1}
\gamma(E)\approx \sup\bigl\{\sigma(E):\sigma\in L_1(E),\,c^2(\sigma)\leq \sigma(E)\bigr\},
\end{equation}
where $L_1(E)$ stands for the set of positive Borel measures supported on $E$ satisfying $\sigma(B(x,r))\leq r$
for all $x\in E$, $r>0$.

Let $\mu$ be a measure supported on $E$ such that $0 < \int_I \|P_\theta\mu\|_2^2\,d\theta<\infty$, and denote
$$\lambda = \frac1{\mu(E)}\int_I \|P_\theta\mu\|_2^2\,d\theta.$$
We intend to construct a suitable measure $\sigma$ with linear growth from $\mu$, and then we will apply \rf{eqgam1} to 
$\sigma$.

Let $\theta_0\in I$ be such that 
$$\|P_{\theta_0}\mu\|_2^2\leq \frac1{\HH^1(I)}\int_I \|P_\theta\mu\|_2^2\,d\theta = \frac{\lambda\,\mu(E)}{\HH^1(I)}.$$
Denote $\eta=P_{\theta_0}\mu$ and let $L_{\theta_0}=\{re^{i\theta_0}:r\in\R\}$. Observe that the preceding estimate is equivalent
to
$$\int_{L_{\theta_0}} \left|\frac{d\eta}{dr}\right|^2\,dr = \int_{P_{\theta_0}(E)} \frac{d\eta}{dr}\,d\eta(r)\leq\frac{\lambda\,\eta(P_{\theta_0}(E))}{\HH^1(I)}.$$
So by Chebyshev's inequality we have 
$$\eta\Bigl\{r \in L_{\theta_0}:\frac{d\eta}{dr}(r) > \frac{2\,\lambda}{\HH^1(I)}\Bigr\}
\leq \frac12\,\eta(P_{\theta_0}(E)).$$
Hence there exists a compact set $F_0$  contained in 
$$\Bigl\{r \in L_{\theta_0}:\frac{d\eta}{dr}(r) \leq \frac{2\,\lambda}{\HH^1(I)}\Bigr\}$$
such that $\eta(F_0)\geq \frac14\,\eta(P_{\theta_0}(E)) = \frac14\,\mu(E)$. Clearly, 
$$\eta(F_0\cap B(x,s))\leq \frac{4\,\lambda}{\HH^1(I)}\,s\quad \mbox{ for all $x\in\C$, $s>0$.}$$ 

Next we consider the closed set $F= P_{\theta_0}^{-1}(F_0)\cap\supp\mu$, and the measure
$$
\sigma= \frac{\HH^1(I)}{4\,\lambda}\,\mu|_F.
$$
Note that 
\begin{equation}\label{eqsigmu4}
\sigma(F) = \frac{\HH^1(I)}{4\,\lambda}\,\mu(F) = \frac{\HH^1(I)}{4\,\lambda}\,\eta(F_0)\geq \frac{\HH^1(I)}{16\,\lambda}\,\mu(E).\end{equation}
Further,
for any $x\in\supp\sigma$ and $s>0$,
$$\sigma(B(x,s))\leq \sigma\bigl(P_{\theta_0}^{-1}(P_{\theta_0}(B(x,s)))\bigr) = \frac{\HH^1(I)}{4\,\lambda}\,
\eta(P_{\theta_0}(F\cap B(x,s)))\leq s,$$
and so $\sigma$ has linear growth with constant $1$.
Also, by the definition of $\lambda$ and \rf{eqsigmu4},
\begin{align*}
\int_{I} \|P_\theta\sigma\|_2^2\,d\theta &= \left(\frac{\HH^1(I)}{4\,\lambda}\right)^2\int_{I} \|P_\theta(\mu|_F)\|_2^2\,d\theta\\
& \leq \left(\frac{\HH^1(I)}{4\,\lambda}\right)^2\int_{I} \|P_\theta\mu\|_2^2\,d\theta = \frac{\HH^1(I)^2}{16\,\lambda}\,\mu(E)
\leq \HH^1(I)\,\sigma(F).
\end{align*}
Hence, by Theorem \ref{teocurv} and Corollary \ref{corofourier}, we deduce that
$$c^2(\sigma)\lesssim \sigma(F) + \iint_{x-y\in K_{I^\bot}}\frac1{|x-y|}\,d\sigma(x)\,d\sigma(y)
 \lesssim \sigma(F) + \int_{I} \|P_\theta\sigma\|_2^2\,d\theta \leq C_I\,\sigma(F),
$$
where the constant $C_I$ depends only on $\HH^1(I)$.
Then, from \rf{eqgam1} and \rf{eqsigmu4}, we deduce that
$$\gamma(E)\geq\gamma(F)\gtrsim \sigma(F) \gtrsim\frac{\mu(E)}{\lambda} = \frac{\mu(E)^2}{\int_I \|P_\theta\mu\|_2^2\,d\theta},$$
with the implicit constants depending on $\HH^1(I)$. This concludes the proof of Theorem \ref{teo1}.

\vv

\subsection{The capacities $\Gamma_{d,n}$}

The proof of Theorem \ref{teo2} is analogous to the one of Theorem \ref{teo1}. The only difference is that we have to
replace the curvature $c^2(\mu)$ by $\sup_{\ve>0}\|\RR^n_\ve\mu\|_{L^2(\mu)}^2$, and use Theorem \ref{teokap} and Corollary
\ref{corofourier2} instead of Theorem \ref{teogam} and Corollary \ref{corofourier}, respectively. We skip the details.

\appendix
\section{The reverse inequality}\label{section:reverse-ineq}

If $\mu$ is a Radon measure on $\R^d$, we define
$$\Theta^{n,*}(x,\mu) = \limsup_{r\to0}\frac{\mu(B(x,r))}{r^n} 
\qquad\text{and}\qquad M_n\mu(x) = \sup_{r>0}\frac{\mu(B(x,r))}{r^n}.$$
In this appendix, we prove the following inequality, which is not used in the paper but may be of independent interest.

\begin{lemma}
Let $V_0 \in G(d, n)$ and $s > 0$. Then there exist $\lambda>1$ and $c$ (depending on $d, n, s$) such that the following holds: If $\mu$ is a  Radon measure on $\R^d$ such that
\begin{equation}\label{eqins5'}
\int M_n\mu(x)\,d\mu(x)<\infty,
\end{equation}
then
\begin{equation}\label{eqins4'}
\int_{B(V_0, s)} \| P_V\mu \|_2^2 \, d \gamma_{d,n}(V) \leq c
 \iint_{x-y \in K(V_0^\bot,\lambda s)}\frac{ d\mu (x)\,  d\mu (y)}{|x-y|^n} + 
c \int \Theta^{n,*}(x,\mu) \,d\mu(x)
.
\end{equation}
\end{lemma}

\begin{proof}
Fix $\phi : \R^d \to \R$ a $C^\infty$ radial bump function with support in $B(0,1)$. For $\ve>0$, let $\phi_\ve(x) = \frac{1}{\ve^d} \phi(\frac{x}{\ve})$.
Denote $\mu_\ve = \mu * \phi_\ve$.
By Corollary \ref{corofourier1}, there exists $\lambda, c > 1$ such that
\begin{equation}\label{eqapp22}
\int_{B(V_0, s)} \| P_V\mu_\ve \|_2^2 \, d \gamma_{d,n}(V)\\  \leq c\iint_{x-y \in K(V_0^\bot,\lambda s)}
\frac{d\mu_\ve(x) \, d\mu_\ve(y)}{|x-y|^n}
.
\end{equation}


First, we will prove that, for any $\ve>0$,
\begin{equation}\label{eqff81}
\iint_{x-y \in K(V_0^\bot,\lambda s)}
\frac{d\mu_\ve(x) \, d\mu_\ve(y)}{|x-y|^n}
\leq 
2
\iint_{x-y \in K(V_0^\bot,2\lambda s)}
\frac{d\mu(x) \, d\mu(y)}{|x-y|^n}
 + 
C\, \int \frac{\mu(B(x,C\ve))}{\ve^n} \,d\mu(x).
\end{equation}
(Here and in what follows, $C$ is independent of $\ve$, and it may change from line to line.) 

For fixed constants $\ve>0$ and $A>10$, denote 
$$f(y)=\chi_{K(V_0^\bot,\lambda s)}(y)\,\frac1{|y|^n},\qquad f_C(y) = f(y)\,\chi_{|y|\leq A\ve},
\qquad f_F(y) = f(y)\,\chi_{|y|> A\ve},
$$
so that 
\begin{align*}
\iint_{x-y \in K(V_0^\bot,\lambda s)}
\frac{d\mu_\ve(x) \, d\mu_\ve(y)}{|x-y|^n}
 & = \int f*\mu_\ve\,d\mu_\ve = \int 
f*\mu*\phi_\ve*\phi_\ve\,d\mu \\
& = \int 
(f_C+f_F)*\mu*\phi_\ve*\phi_\ve\,d\mu.
\end{align*}
We will show that, for any $x\in\R^{d}$,
\begin{equation}\label{eqff91}
f_C*\mu*\phi_\ve*\phi_\ve(x)\leq C(A)\,\frac{\mu(B(x,C(A)\ve))}{\ve^n}
\end{equation}
and
\begin{equation}\label{eqff92}
f_F*\mu*\phi_\ve*\phi_\ve(x)\leq 2 \int_{x-y\in K(V_0^\bot,2\lambda s)} \frac{d\mu(y)}{|x-y|^n},
\end{equation}
if $A$ is sufficiently large. Clearly, \rf{eqff81} follows from the two preceding estimates.

First we deal with the estimate \rf{eqff91}. Let $\psi = \phi * \phi$ and $\psi_\ve(x) = \frac{1}{\ve^d}\psi(\frac{x}{\ve})$, so that $\psi_\ve = \phi_\ve*\phi_\ve$.
For any $z\in\R^d$,
\begin{align}\label{eq:fC-conv}
f_C*\psi_\ve(z) 
\leq 
\int_{|z-y|\leq A\ve} \frac1{|z-y|^n}\,\psi_\ve(y)\,dy
\leq 
\frac{C(A)}{\ve^n}\,\chi_{B(0,(A+2)\ve)}(z)
,
\end{align}
taking into account that $\supp\psi_\ve\subset B(0,2\ve)$ and $\| \psi_\ve \|_\infty \leq c_{\phi} \, \ve^{-d}$. It is clear that \eqref{eq:fC-conv} implies \rf{eqff91}.

To prove \rf{eqff92}, first observe that since $\supp f_F \subset K(V_0^\bot,\lambda s) \cap B(0, A\ve)^c$ and $\supp \psi_\ve \subset B(0, 2\ve)$, it follows that $\supp (f_F * \psi_\ve)$ is contained in the $2\ve$-neighborhood of $K(V_0^\bot,\lambda s) \cap B(0, A\ve)^c$. Therefore, by geometric arguments, we have
\[
\supp f_F * \psi_\ve \subset
K(V_0^\bot,\lambda s + CA^{-1})\subset K(V_0^\bot,2\lambda s )
\]
assuming $A$ big enough.

Next, suppose $x \in \supp (f_F * \psi_\ve)$ and $x' \in B(x, 2\ve)$. Since $|x| \geq (A-2)\ve$, we have $|x'|\geq \frac12\,|x|$ so $f_F(x') \leq \frac{2}{|x|^n}$. Hence, for all $x \in \R^d$,
\[
f_F*\psi_\ve(x)
\leq 
\sup_{x'\in B(x,2\ve)} f_F(x')
\leq
2\frac{\chi_{K(V_0^\bot,2\lambda s)}(x)}{|x|^n}
,
\]
which yields \rf{eqff92}, and completes the proof of \rf{eqff81}.

\vv

By Fatou's lemma applied to $M_n(x) - \frac{\mu(B(x,C\ve))}{\ve^n}$ and hypothesis \rf{eqins5'}, we have
$$\limsup_{\ve\to0}\int \frac{\mu(B(x,C\ve))}{\ve^n}\,d\mu(x)\leq 
\int \limsup_{\ve\to0}\frac{\mu(B(x,C\ve))}{\ve^n}\,d\mu(x) = C\int \Theta^{n,*}(x,\mu)\,d\mu(x).$$
Taking the limsup in \rf{eqff81} as $\ve\to0$ and using \rf{eqapp22}, we obtain 
\begin{equation}\label{eqapp1}
\limsup_{\ve\to 0} 
\int_{B(V_0, s)} \| P_V\mu_\ve \|_2^2 \, d \gamma_{d,n}(V)
\leq
C
 \iint_{x-y \in K(V_0^\bot,2\lambda s)}\frac{ d\mu (x)\,  d\mu (y)}{|x-y|^n} + 
C \int \Theta^{n,*}(x,\mu) \,d\mu(x).
\end{equation}

\vv
Now we claim that
\begin{equation}\label{eqapp93}
\int_{B(V_0, s)} \| P_V\mu \|_2^2 \, d \gamma_{d,n}(V) 
= 
\lim_{\ve\to0}
\int_{B(V_0, s)} \| P_V\mu_\ve \|_2^2 \, d \gamma_{d,n}(V)
.
\end{equation}
Note that \eqref{eqapp1} and \eqref{eqapp93} together imply \rf{eqins4'} with $2\lambda$ in place of $\lambda$.

Let $\sigma$ be the measure on $\R^d$ given by
\begin{align*}
\int f \, d\sigma
&=
\int_{B(V_0,s)} \int_V f \, d\LL^n \, d\gamma_{d,n}(V),
\end{align*}
Then arguing analogously as in the proof of \rf{eq843}, we have
\begin{equation}\label{eqffoo1}
\int_{B(V_0, s)} \| P_V\mu_\ve \|_2^2 \, d \gamma_{d,n}(V) 
=   \int |\wh\mu(x)\,\wh\phi(\ve x)|^2 \, d\sigma(x),
\end{equation}
and
\begin{equation}\label{eqffoo2}
\int_{B(V_0, s)} \| P_V\mu \|_2^2 \, d \gamma_{d,n}(V) 
=   \int |\wh\mu(x)|^2 \, d\sigma(x).
\end{equation}

We split the proof of \eqref{eqapp93} into two cases. Suppose first that $\int_{B(V_0, s)} \| P_V\mu \|_2^2 \, d \gamma_{d,n}(V)  < \infty$. In this case, the dominated convergence theorem gives us
\begin{align*}
\lim_{\ve\to0} 
\int_{B(V_0, s)} \| P_V\mu_\ve \|_2^2 \, d \gamma_{d,n}(V) 
&=   
\lim_{\ve\to0} 
\int |\wh\mu(x)\,\wh\phi(\ve x)|^2 \, d\sigma(x)
\\
& =
   \int |\wh\mu(x)|^2 \, d\sigma(x)
    =
\int_{B(V_0, s)} \| P_V\mu \|_2^2 \, d \gamma_{d,n}(V) 
,
\end{align*}
which proves \eqref{eqapp93} in this case.

Now we consider the case $\int_{B(V_0, s)} \| P_V\mu \|_2^2 \, d \gamma_{d,n}(V) = \infty$. By Fatou's lemma,
\begin{align*}
\liminf_{\ve\to0} 
\int_{B(V_0, s)} \| P_V\mu_\ve \|_2^2 \, d \gamma_{d,n}(V) 
&=   
\liminf_{\ve\to0} 
\int |\wh\mu(x)\,\wh\phi(\ve x)|^2 \, d\sigma(x)
\\
& \geq
\int \liminf_{\ve\to0}  |\wh\mu(x)\,\wh\phi(\ve x)|^2 \, d\sigma(x)
\\
&=
\int |\wh\mu(x)|^2 \, d\sigma(x)
\\
&=
\int_{B(V_0, s)} \| P_V\mu \|_2^2 \, d \gamma_{d,n}(V)
\\
&=
\infty 
,
\end{align*}
which proves \eqref{eqapp93} in this case. This completes the proof of \eqref{eqapp93}.
\end{proof}

\vvv

\end{document}